\documentclass[12pt]{article}
\usepackage{epsfig,psfrag,amsmath,amssymb,latexsym}
\usepackage{amscd }
\usepackage{color}
\usepackage{amsfonts}
\usepackage{graphicx}
\usepackage{wasysym}
\usepackage{mathrsfs}
\numberwithin{equation}{section}
\newtheorem{thm}{Theorem}[section]

\newtheorem{theorem}[thm]{Theorem}
\newtheorem{corollary}[thm]{Corollary}
\newtheorem{fact}[thm]{Fact}
\newtheorem{definition}[thm]{Definition}
\newtheorem{lemma}[thm]{Lemma}
\newenvironment{proof}[1][Proof]{\textbf{#1.} }{\ \rule{0.5em}{0.5em}}

\newcommand{\htwo}{{H^{2|2}}}

\newcommand{\ka}{\kappa}
\newcommand{\al}{\alpha}
\newcommand{\be}{\beta}
\newcommand{\lastexit}{{\rm last\, exit}}
\newcommand{\id}{\operatorname{id}}
\newcommand{\Id}{\operatorname{Id}}
\newcommand{\susy}{\mathrm{susy}}
\newcommand{\bigsusy}{{\mathrm{big}}}
\newcommand{\single}{{\mathrm{single}}}
\newcommand{\supp}{\operatorname{supp}}
\newcommand{\discrete}{\operatorname{discrete}}
\newcommand{\K}{\mathcal{K}}

\newcommand{\pp}{\mathscr{P}}
\newcommand{\cL}{\mathcal{L}}
\newcommand{\cO}{\mathcal{O}}
\newcommand{\cQ}{\mathcal{Q}}
\newcommand{\cH}{\mathcal{H}}
\newcommand{\V}{{\mathscr{V}}}
\newcommand{\cV}{\mathcal{V}}
\newcommand{\R}{{\mathbb R}}  
\newcommand{\N}{{\mathbb N}}  
\newcommand{\Z}{{\mathbb Z}}  
\newcommand{\F}{{\mathcal{F}}}
\newcommand{\T}{{\mathcal{T}}}
\newcommand{\G}{{\mathcal{G}}}

\begin{document}
\thispagestyle{empty}

\begin{center}
{\LARGE Convergence of vertex-reinforced jump processes\\[-1mm]
to an extension of the \\[2mm] 
supersymmetric hyperbolic nonlinear sigma 
model}\\[3mm]
{Franz Merkl\footnote{Mathematical Institute, Ludwig-Maximilians-Universit\"at M\"unchen,
Theresienstr.\ 39,
D-80333 Munich,
Germany.
E-mail: merkl@math.lmu.de
}
\hspace{1cm} 
Silke W.W.\ Rolles\footnote{Zentrum Mathematik, Bereich M5,
Technische Universit{\"{a}}t M{\"{u}}nchen,
D-85747 Garching bei M{\"{u}}nchen,
Germany.
E-mail: srolles@ma.tum.de}
\hspace{1cm} 
Pierre Tarr{\`e}s\footnote{NYU-ECNU 
Institute of Mathematical Sciences at NYU Shanghai, 
Courant Institute of Mathematical Sciences, New York, 
CNRS and Universit\'e Paris-Dauphine, 
PSL Research University, Ceremade,
75016 Paris, France.
E-mail: tarres@nyu.edu}
\\[3mm]
{\small \today}}\\[3mm]
\end{center}

\begin{abstract}
In this paper, we define an extension of the supersymmetric hyperbolic 
nonlinear sigma model introduced by Zirnbauer. We show that it arises
as a weak joint limit of a time-changed version introduced by 
Sabot and Tarr\`es of the vertex-reinforced jump process. 
It describes the asymptotics of rescaled crossing numbers, rescaled 
fluctuations of local times, asymptotic local times on a 
logarithmic scale, endpoints of paths, and last exit trees. 
\end{abstract}

\section{Introduction and results}

\subsection{Extension of the susy hyperbolic nonlinear sigma model}

The supersymmetric hyperbolic nonlinear sigma model, called $\htwo$ model
for short, was introduced by Zirnbauer in \cite{zirnbauer-91}. Concerning 
its original motivation, Zirnbauer writes that it may serve as a 
toy model for studying diffusion and localization in disordered one-electron systems. 
The $\htwo$ model is a statistical mechanics type model defined over a 
finite undirected graph $G=(V,E)$. Any undirected edge $\{i,j\}\in E$ is 
given a weight $W_{ij}=W_{ji}>0$. In its original form, which is not used
in this paper, the ``spin variables''
at any vertex take their value in a supermanifold $\htwo$ having the hyperbolic
plane $H^2$ as its base manifold. Written in so-called ``horospherical 
coordinates'', the model associates to any vertex $i\in V$ two kinds
of ``spin variables'': two real-valued variables $s_i$ and $u_i$ and two 
Grassmann (anticommuting) variables $\overline\psi_i$ and $\psi_i$.  
In the description with Grassmann variables, the model has useful 
supersymmetries as is shown in 
the paper \cite{disertori-spencer-zirnbauer2010} by Disertori, Spencer, and 
Zirnbauer; note that $u_i$ is called $t_i$ in that paper. 

However, in the current paper we use an equivalent purely probabilistic 
description of the $\htwo$ model where the Grassmann variables $\overline\psi_i$ 
and $\psi_i$ are replaced by a discrete variable $T'$ taking values 
in the set $\T$ of undirected spanning trees of $G$. Any undirected spanning
tree is viewed as a set of undirected edges. The tree variant of the $\htwo$ 
model was e.g.\ also used in \cite{disertori-merkl-rolles16-comparison}, 
formula (2.6). 
It has the disadvantage that the supersymmetries become hidden, but the 
advantage that it is phrased solely in probabilistic terms. 
It is defined as follows. 
Given a fixed reference point 
$i_0\in V$, the vectors $s=(s_i)_{i\in V}$ and $u=(u_i)_{i\in V}$ take values
in the set 
\begin{align}
\label{eq:def-Omega}
\Omega_{i_0}=\{u\in\R^V:u_{i_0}=0\}.
\end{align}

\begin{definition}{\bf (Tree version of the $\htwo$ model)}
\label{def:tree-version-H22}
The tree version of the supersymmetric hyperbolic nonlinear sigma model
is the following probability measure on $\Omega_{i_0}^2\times\T$:
\begin{align}
\mu_{i_0}^\susy(ds\, du\, dT') 
=&  
\exp\Big(\sum_{\{i,j\}\in E} W_{ij}\Big(
1-\cosh(u_i-u_j)-\frac12e^{u_i+u_j}(s_i-s_j)^2\Big)\Big) \nonumber\\
& \cdot \prod_{\{i,j\}\in T'}W_{ij}e^{u_i+u_j}
\prod_{i\in V\setminus\{i_0\}}\frac{e^{-u_i}\,ds_i\, du_i}{2\pi} 
 \cdot dT', 
\label{eq:marginal2-rho-susy}
\end{align}
where $ds_i$ and $du_i$ denote the Lebesgue measure on $\R$, and $dT'$ means 
the counting measure on $\T$. 
\end{definition}

It is a non-trivial fact that $\mu_{i_0}^\susy$ is a probability measure, i.e.
\begin{align}
\label{eq:normalizing-const}
\mu_{i_0}^\susy(\Omega_{i_0}^2\times\T)=1.
\end{align} 
There are at least three different proofs of this fact, given in
\cite{disertori-spencer-zirnbauer2010}, \cite{sabot-tarres2012}, and 
\cite{sabot-tarres-zeng2015}. Precise references and some more comments on this 
are given in appendix \ref{sec:app-review-results}, which reviews known results 
used here.

\paragraph{Aim of this paper.}
The main goal of the present paper is to give an interpretation of all 
random variables $s,u,T'$, jointly distributed according to $\mu_{i_0}^\susy$, in 
terms of limits of vertex-reinforced jump processes. For linearly edge-reinforced random 
walks, which are processes in discrete time, a similar asymptotic analysis 
was given in \cite{Keane-Rolles2000}; see also Theorem 3.2 in 
\cite{merkl-rolles-festschrift-2006}. However, due to continuous time, the 
analysis in the current paper requires additional considerations, in particular, 
when dealing with local times and their fluctuations. 
In the present setup, even more random variables than only $s,u$, and $T'$ 
occur naturally: $v=(v_i)_{i\in V}\in\Omega_{i_0}$, $i_1\in V$,
$i_1'\in V$, a second spanning tree $T$ playing a similar role as $T'$, and 
two vectors $\kappa$ and $\kappa'$. More precisely, one may view 
$\kappa$ and $\kappa'$ as currents flowing through the edges of the 
graph. They take values in the space $\cH$ of sourceless currents, defined as 
follows: Let $\vec{E}=\{(i,j):\{i,j\}\in E\}$ denote the set of directed 
edges, where each undirected edge in $E$ is replaced
by two directed edges with opposite directions. 
Let $\cH$ denote the linear subspace of $\R^{\vec E}$ consisting of all 
$\kappa=(\kappa_{ij})_{(i,j)\in\vec E}$ that satisfy the homogeneous Kirchhoff rules given by 
\begin{align}
\label{eq:Kirchhoff-kappa}
\sum_{\substack{j\in V:\\ \{i,j\}\in E}}(\kappa_{ij}-\kappa_{ji})=0\quad\text{for all }i\in V.
\end{align}
We endow $\cH$ with the Lebesgue measure $d\kappa_\cH$ defined as follows: 
Take any directed reference spanning tree $\vec T_0$ of $G$. Note that the restriction 
$\iota:\cH\to\R^{\vec E\setminus\vec T_0}$
of the restriction map $\R^{\vec E}\to\R^{\vec E\setminus\vec T_0}$ is an isomorphism. 
Let $d\kappa_\cH$ be the image under $\iota^{-1}$ of the Lebesgue measure on 
$\R^{\vec E\setminus\vec T_0}$. Note that $d\kappa_\cH$ does not depend on the 
choice of $\vec T_0$. For $\{i,j\}\in E$, we abbreviate
\begin{align}
\label{eq:def-omega}
\omega_{ij}= \frac{W_{ij}}{2}e^{v_i+v_j}, \quad 
\omega_{ij}'= \frac{W_{ij}}{2}e^{u_i+u_j}.
\end{align}
Using this, we define an extended version of the supersymmetric hyperbolic 
nonlinear sigma model that involves not only the original variables $s,u$, and 
$T'$, but also the new variables $\kappa,\kappa',v,i_1,i_1'$, and $T$.

\begin{definition}{\bf (An extended version of the $\htwo$ model)}
\label{def:h2-extended}
We define the function 
$\rho^\bigsusy_{i_0}:(\R^{\vec E})^2\times\Omega_{i_0}^3\times V^2\times\T^2\to\R^+$
by 
\begin{align}
\label{eq:def-rho-susy}
& \rho^\bigsusy_{i_0}=\rho^\bigsusy_{i_0}(\kappa,\kappa',s,v,u,i_1,i_1',T,T')\\
=&  \frac{4^{|V|-1}}{(2\pi)^{2|E|}}
\exp\Big(\sum_{\{i,j\}\in E} W_{ij}\Big(
1-\cosh(u_i-u_j)-\frac12e^{u_i+u_j}(s_i-s_j)^2\Big)\Big) 
\cdot \prod_{\{i,j\}\in T}\omega_{ij}' \nonumber\\
& \cdot\prod_{\{i,j\}\in T'}\omega_{ij}'\cdot\!
\prod_{\{i,j\}\in E}\frac{1}{(\omega_{ij}')^2}
\cdot\exp\Big(
-\sum_{(i,j)\in\vec E} \frac{\kappa_{ij}^2+(\kappa_{ij}')^2}{2\omega_{ij}'}
\Big) 
\frac{e^{2v_{i_1}+2u_{i_1'}}}{\sum_{i\in V}e^{2 v_i}\sum_{j\in V}e^{2 u_j}}
\prod_{i\in V\setminus\{i_0\}}e^{-u_i}.
\nonumber
\end{align}
Furthermore, we define the measure 
\begin{align}
\label{eq:def-mu-susy}
&\mu^\bigsusy_{i_0}(d\kappa\, d\kappa'\, ds\, du\, dv\, di_1\, di_1'\, dT\, dT')\cr
=&  \rho_{i_0}^\bigsusy \, d\kappa_\cH\, d\kappa_\cH' 
\prod_{i\in V\setminus\{i_0\}} 1_{\{u_i=v_i\}}\, ds_i\, du_i
\cdot di_1\, di_1'\, dT\, dT'
\end{align}
on $\cH^2\times\Omega_{i_0}^3\times V^2\times\T^2$, where in the last
expression $1_{\{u_i=v_i\}}du_i$ denotes the Lebesgue measure on the diagonal of $\R^2$, 
$di_1,di_1'$ denote the counting measure on $V$ and $dT,dT'$ mean the counting measure
on $\T$. 
\end{definition}

The reader might wonder why we introduce a seemingly redundant variable 
$v=u$. The reason becomes clear below when we describe the asymptotics 
of vertex-reinforced jump processes; cf.\ the considerations following \eqref{eq:def-u-v} 
and Theorem \ref{thm:claim-thm-asymptotics-collected}. 

The following theorem explains the link between the extended $\htwo$ 
model and the tree version of the $\htwo$ model.

\begin{theorem}{\bf (The $\htwo$ model as a marginal of its extended version)}
\label{thm:marginal-rho-susy}\\
The marginal of $(s,u,T')$ with respect to $\mu_{i_0}^\bigsusy$ 
equals the tree version of the non-linear supersymmetric sigma model: 
\begin{align}
\cL_{\mu_{i_0}^\bigsusy}(s,u,T') = \mu_{i_0}^\susy.
\label{eq:marginal2-big-susy}
\end{align}
In particular, $\mu_{i_0}^\bigsusy$ is a probability measure. 
\end{theorem}

In the interpretation of the extended $\htwo$ model in terms of vertex-reinforced
jump processes explained in the next subsection, two different time scales 
$\sigma\ll\sigma'$ play a role.
Then, the list of variables $\kappa,\kappa',s,v,u,i_1,i_1',T,T'$ splits into 
two groups: $\kappa,v,i_1$, and $T$ involve only the first time scale $\sigma$, 
whereas the remaining variables $\kappa',s,u,i_1'$, and $T'$ involve both 
time scales. This is why we are also interested in the following marginal of 
the extended $\htwo$ model.

\begin{theorem}{\bf (Single-time-scale marginal)}
\label{thm:second-marginal-rho-susy}
The marginal of $(\kappa,v,i_1,T)$ with respect to $\mu_{i_0}^\bigsusy$ 
equals the following probability measure $\mu_{i_0}^\single$
on $\cH\times\Omega_{i_0}\times V\times\T$:
\begin{align}
\label{eq:def-mu-single}
\mu_{i_0}^\single(d\kappa\, dv\, di_1\, dT) =\rho_{i_0}^\single
d\kappa_\cH\prod_{i\in V\setminus\{i_0\}} dv_i\cdot di_1\, dT
\end{align}
with the density 
\begin{align}
\rho_{i_0}^\single(\kappa,v,i_1,T) 
= & \frac{1}{\pi^{|E|}}
\exp\Big(\sum_{\{i,j\}\in E} W_{ij}(
1-\cosh(v_i-v_j))\Big) 
\cdot\prod_{\{i,j\}\in E\setminus T}\frac{1}{W_{ij}e^{v_i+v_j}}\nonumber\\
& \cdot\exp\Big(
-\sum_{(i,j)\in\vec E} \frac{\kappa_{ij}^2}{W_{ij}e^{v_i+v_j}}
\Big) 
\frac{e^{2v_{i_1}}}{\sum_{i\in V}e^{2 v_i}}
\prod_{i\in V\setminus\{i_0\}}e^{-v_i}. 
\label{eq:def-rho-single}
\end{align}
\end{theorem}

\subsection{Vertex-reinforced jump processes}

Consider again a finite undirected graph $G$
with edge weights $W=(W_{ij})_{i,j\in V}$, $W_{ij}=W_{ji}$, where 
we set $W_{ij}=0$ for $i,j\in V$ with $\{i,j\}\not\in E$ to simplify notation. 
The vertex-reinforced jump process (VRJP) is a stochastic jump process $Y=(Y_t)_{t\ge 0}$ 
in continuous time with c\`adl\`ag paths, taking values in the vertex set $V$ 
of $G$. The process starts in $Y_0=i_0\in V$. Let $P_{i_0}$ denote the underlying
probability measure. The jump rates are defined in terms of the local
times with offset 1 given by 
\begin{align}
L_i(t)=1+\int_0^t 1_{\{Y_\tau=i\}}\, d\tau, \quad i\in V,
\end{align}
at times $t\ge 0$. In other words, the local time $L_i(t)$ is 1 plus the time 
the process $Y$ spends in vertex $i$ up to time $t$. 
Given two different vertices $i,j\in V$, 
a time $t\ge 0$, and another (small) time $\Delta t>0$, on the event 
$\{Y_t=i\}$
and conditionally on the past $\F_t=\sigma(Y_\tau: 0\le\tau\le t)$ up to time $t$, 
its jump probability is given by 
\begin{align}
\label{eq:jump-rate-Y}
P_{i_0}(Y_{t+\Delta t}=j|\F_t,Y_t=i)
=W_{ij} L_j(t) \Delta t+o(\Delta t)
\quad P_{i_0}\text{-a.s.\ as }\Delta t\downarrow 0. 
\end{align}
In other words, the process has the jump rates $W_{ij} L_j(t)$.

In the following we consider the time-changed version of the vertex-reinforced 
jump process $Z=(Z_\sigma)_{\sigma\ge 0}$ on $G$ which was introduced in 
\cite{sabot-tarres2012}. Let us review its definition. The time change is defined by 
\begin{align}
\label{eq:def-D}
D(t)=\sum_{i\in V} (L_i(t)^2-1). 
\end{align}
The time-changed version $Z$ is defined by 
\begin{align}
Z_\sigma=Y_{D^{-1}(\sigma)},\quad \sigma\ge 0.
\end{align}
The local time $l(\sigma)=(l_i(\sigma))_{i\in V}$ (without offset) 
of the process $Z$ is defined by 
\begin{align}
l_i(\sigma)=\int_0^\sigma1_{\{Z_{\zeta}=i\}}\, d\zeta
=L_i(D^{-1}(\sigma))^2-1, \quad i\in V.
\end{align}
The second equality in the last display is just another way of writing 
equation (4.13) in \cite{sabot-tarres2012}.
The jump rates of the process $Y$ specified in \eqref{eq:jump-rate-Y}
are transformed by the time-change as follows. 
Given two different vertices $i,j\in V$, 
a time $\sigma\ge 0$, and another (small) time $\Delta\sigma>0$, on the event 
$\{Z_\sigma=i\}$
and conditionally on the past $\G_\sigma=\sigma(Z_\zeta: 0\le\zeta\le \sigma)$ up to 
time $\sigma$, its jump probability is given by 
\begin{align}
P_{i_0}(Z_{\sigma+\Delta\sigma}=j|\G_\sigma,Z_\sigma=i)
=\frac{W_{ij}}{2}\sqrt{\frac{1+l_j(\sigma)}{1+l_i(\sigma)}}\Delta\sigma+o(\Delta\sigma)
\quad P_{i_0}\text{-a.s.\ as }\Delta\sigma\downarrow 0. 
\end{align}
In other words, the process $Z$ has the jump rates 
$\frac{W_{ij}}{2}\sqrt{\frac{1+l_j(\sigma)}{1+l_i(\sigma)}}$. 

\paragraph{History of the model.} 
The vertex-reinforced jump process was initially proposed by Werner and 
introduced by Davis and Volkov in \cite{davis-volkov1} on the integers and 
studied in \cite{davis-volkov2} on trees. Further analysis on 
regular and Galton-Watson trees was conducted by 
Collevecchio in \cite{collevecchio2} and \cite{collevecchio1} and by
Basdevant and Singh in \cite{basdevant-singh2012}. 

Tarr\`es in \cite{tarres2011} and 
Sabot and Tarr\`es in \cite{sabot-tarres2012} showed that
the VRJP is related to the linearly edge-reinforced random walk. 
In \cite{sabot-tarres2012}, Sabot and Tarr\`es also showed that the 
VRJP is associated to the 
$u$-marginal of the supersymmetric hyperbolic sigma model $\htwo$ defined above;
cf.\ formula \eqref{eq:u-as-limit-for-vrjp}, below. 
Using this second link and the results from \cite{disertori-spencer-zirnbauer2010}
and \cite{disertori-spencer2010}, they proved in the same paper recurrence of VRJP on any graph 
of bounded degree for strong reinforcement, i.e.\ $W_{ij}$ small for all edges $\{i,j\}$, 
and transience on $\Z^d$, $d\ge 3$, for small constant reinforcement, i.e.\ $W_{ij}$ large 
and constant for all edges. 
Disertori, Sabot, and Tarr\`es \cite{disertori-sabot-tarres2014} give a generalization 
to non-constant reinforcement. An alternative proof of recurrence of VRJP under the same 
conditions, not using the connection with the $\htwo$ model, was given by 
Angel, Crawford, and Kozma in \cite{angel-crawford-kozma}. 
In \cite{disertori-merkl-rolles2014}, Disertori, Merkl, and Rolles show recurrence of 
VRJP on two-sided infinite strips with translationally invariant $W_{ij}$. 

Further links of VRJP with a random Schr\"odinger operator and with Ray-Knight 
second generalized theorem were investigated by Sabot, Tarr\`es, and Zeng in 
\cite{sabot-tarres-zeng2015}, \cite{sabot-zeng15}, and \cite{sabot-tarres2016}.

\paragraph{Current setup.}
The components $u$ in the $\htwo$ model have the following interpretation 
in terms of VRJP:  
Proposition 1 and Theorem 2 in \cite{sabot-tarres2012} imply that 
$l_i(\sigma)/l_{i_0}(\sigma)$, $i\in V$, converge jointly $P_{i_0}$-almost surely 
to a limit having the law 
\begin{align}
\label{eq:u-as-limit-for-vrjp}
\cL_{P_{i_0}}\left(\lim_{\sigma\to\infty}\left(\frac{l_i(\sigma)}{l_{i_0}(\sigma)}\right)_{i\in V}
\right)
=\cL_{\mu_{i_0}^\susy}\left(\left(e^{2u_i}\right)_{i\in V}\right)
\end{align}
with the tree version $\mu_{i_0}^\susy$ of the $\htwo$ model given in Definition 
\ref{def:tree-version-H22}. One of the goals of the current paper is to show 
that \emph{all} the components $\kappa, \kappa', s, u, v, i_1, i_1', T$, and $T'$
in the extended $\htwo$ model have an interpretation in 
terms of VRJP as well. We shall show below that the components $s_i$ can be
interpreted in terms of fluctuations of $l_i(\sigma)/l_{i_0}(\sigma)$ around 
its asymptotic value $\lim_{\sigma'\to\infty}l_i(\sigma')/l_{i_0}(\sigma')$. 
For this analysis, it is natural to consider two different 
time scales $1\ll\sigma\ll\sigma'$. 
We remark that introducing further time scales 
would not be relevant since it would only add conditionally independent 
copies of the same variables $s$, $\kappa$ and $T$. 
To make the result more accessible, we start in Section 
\ref{sec:single-time-scale} with single timescale 
versions of our statements which will be used to obtain directly the 
$\mu_{i_0}^\single$ marginal of Theorem \ref{thm:second-marginal-rho-susy}.

\subsection{Single timescale limit of VRJP}
\label{sec:single-time-scale}

\paragraph{Last exit trees.}
We shall show below that the component $T'$ in the $\htwo$ model 
has an interpretation in terms of last exit trees of the process $Z$. 
These last exit trees are defined as follows. 
Given a time interval $[\sigma_1,\sigma_2]$, let 
$V_{[\sigma_1,\sigma_2]}=\{Z_t:t\in[\sigma_1,\sigma_2]\}$
denote the set of all vertices visited between times $\sigma_1$ and $\sigma_2$. 
For $i\in V_{[\sigma_1,\sigma_2]}\setminus\{Z_{\sigma_2}\}$, let 
$e^\lastexit_i(\sigma_1,\sigma_2)$ denote the directed edge 
of the form $(i,j)$ which the process $Z$ has crossed when it left vertex $i$
for the last time during the time interval $[\sigma_1,\sigma_2]$. Let
\begin{align}
\vec T^\lastexit(\sigma_1,\sigma_2)
=\bigcup_{i\in V_{[\sigma_1,\sigma_2]}\setminus\{Z_{\sigma_2}\}}
\{e^\lastexit_i(\sigma_1,\sigma_2)\}
\end{align}
be the collection of directed edges taken by the process $Z$ for the last departures from 
all vertices visited in the time interval $[\sigma_1,\sigma_2]$ except the endpoint. 
Sometimes, we need also the undirected
version of $\vec T^\lastexit(\sigma_1,\sigma_2)$; it is denoted by 
$T^\lastexit(\sigma_1,\sigma_2)$. More generally, 
whenever we have a directed spanning 
tree $\vec T$, we denote its undirected version by $T$. 
If the process $Z$ has visited all vertices between times $\sigma_1$ and $\sigma_2$, 
then $\vec T^\lastexit(\sigma_1,\sigma_2)$ is a spanning tree of $G$, 
directed towards the endpoint $Z_{\sigma_2}$. For $i_1\in V$ let $\vec\T_{i_1}$ 
denote the set of spanning trees of $G$ which are directed towards $i_1$. 

\paragraph{Edge crossings and currents with sources.}
We define $k(\sigma)=(k_{ij}(\sigma))_{(i,j)\in\vec{E}}$ by 
\begin{align}
k_{ij}(\sigma)=|\{t\le\sigma: Z_{t-}=i,Z_t=j\}|, 
\end{align}
which denotes the number of crossings from $i$ to $j$ up to time $\sigma$. 
We denote by $\delta_i(j)=1_{\{i=j\}}$ Kronecker's delta.
For $i_0,i_1\in V$, let $\K_{i_0,i_1}$ denote the set of all 
$k\in\Z^{\vec E}$ such that the inhomogeneous Kirchhoff rules 
\begin{align}
\label{eq:Kirchhoff}
\sum_{\substack{j\in V:\\ \{i,j\}\in E}}(k_{ij}-k_{ji})=\delta_{i_0}(i)-\delta_{i_1}(i), \quad
i\in V,
\end{align}
hold. One can imagine $k$ as a current flowing through the graph with a source
of size 1 at $i_0$ and a sink of size 1 at $i_1$. 
We are only interested in edge crossings $k$ compatible with at least one
path from $i_0$ to $i_1$, which is guaranteed for $k$ in the set
\begin{align}
\K_{i_0,i_1}^+:=\K_{i_0,i_1}\cap\N^{\vec E}.
\end{align}
However, when comparing edge crossings $k$ and $\pi$ of two paths, 
differences $k-\pi\in\K_{i_0,i_1}\setminus\K_{i_0,i_1}^+$ can occur as well.
Let 
\begin{align}
\label{eq:def-L}
\cL_\sigma=\Big\{ l\in(0,\infty)^V:\sum_{i\in V}l_i=\sigma\Big\}.
\end{align}
For $k\in\N^{\vec E}$, $l\in(0,\infty)^V$, and a directed spanning tree 
$\vec T$, we abbreviate 
\begin{align}
\label{eq:def-pp}
\pp(k,l,\vec T)=
\prod_{(i,j)\in\vec{E}}\left(\frac{W_{ij}l_i}{2}\right)^{k_{ij}}\frac{1}{k_{ij}!}
\prod_{(i,j)\in\vec T}\frac{k_{ij}}{l_i}.
\end{align}

\begin{theorem}{\bf (Joint density of edge crossings, local times, and last exit trees)}
\label{thm:proba-of-a-path-single}
For $i_0,i_1\in V$, $k\in\K_{i_0,i_1}^+$, $\sigma>0$, 
$A\subseteq\cL_\sigma$ measurable, and $\vec T\in\vec\T_{i_1}$, we have 
\begin{align}
&P_{i_0}(k(\sigma)=k,l(\sigma)\in A,\vec T^\lastexit(0,\sigma)=\vec T)
\label{eq:density-path-single}\\
= & \int_A
\exp\Big(
\sum_{\{i,j\}\in E} W_{ij}\left(1-\sqrt{1+l_i}\sqrt{1+l_j}\right)\Big)
\prod_{i\in V\setminus\{i_1\}}\frac{1}{\sqrt{1+l_i}} 
\cdot \pp(k,l,\vec T) \prod_{i\in V\setminus\{i_0\}}dl_i.
\nonumber
\end{align}
\end{theorem}

It follows from Theorem \ref{thm:proba-of-a-path-single} that conditionally 
on $l(\sigma)$ and $\vec T^\lastexit(0,\sigma)$, the random variable $k(\sigma)$
follows a variant of a random current model, i.e.\ a product of Poisson distributions 
conditioned on Kirchhoff's rule. 

By a slight abuse of notation, we will abbreviate henceforth
\begin{eqnarray}
k=k(\sigma)\quad\text{ and }\quad l=l(\sigma).
\end{eqnarray}
On the event $\{l_i>0\text{ for all }i\in V\}$, motivated by 
\eqref{eq:u-as-limit-for-vrjp}, 
we introduce new variables $v_i=v_i(\sigma)$ for $i\in V$ by 
\begin{align}
\label{eq:def-v}
& l_i=l_{i_0}e^{2v_i}.
\end{align}
Since $v_{i_0}=0$, the vector $(v_i)_{i\in V}$ belongs to the
space $\Omega_{i_0}$ defined in \eqref{eq:def-Omega}.

\paragraph{Rescaling of crossing numbers.}
It turns out that the random variables $k_{ij}(\sigma)$ are centered 
roughly around $\frac12W_{ij}\sqrt{l_il_j}$, with fluctuations on the scale
$\sqrt{l_{i_0}}$. This motivates us to introduce, again on the event 
$\{l_i>0\text{ for all }i\in V\}$, the rescaled 
crossing numbers $\kappa_{ij}=\kappa_{ij}(\sigma)$ for $(i,j)\in\vec E$ by 
\begin{align}
& \kappa_{ij}=\frac{k_{ij}-\frac12W_{ij}\sqrt{l_il_j}}{\sqrt{l_{i_0}}}
=\frac{k_{ij}}{\sqrt{l_{i_0}}}-\frac{W_{ij}}{2}e^{v_i+v_j}{\sqrt{l_{i_0}}}. 
\label{eq:def-kappa}
\end{align}
The rescaling with the factor $l_{i_0}^{-1/2}$, which $P_{i_0}$-a.s.\ converges 
to $0$ as $\sigma\to\infty$, makes the sources $\pm 1$ of the current
$k$ in the vertices $i_0$ and $Z_\sigma$ 
asymptotically negligible. This explains intuitively why 
the homogeneous Kirchhoff rules \eqref{eq:Kirchhoff-kappa}
rather than the inhomogeneous Kirchhoff rules \eqref{eq:Kirchhoff} apply 
asymptotically to $\kappa$.

For a truncation parameter $M>0$, we consider the events 
\begin{align}
\label{eq:def-event-B-single}
&B_\sigma(M)=\{ |\kappa_{ij}|,|v_i|\le M, l_i>0  
\text{ for all }i,j\in V\}.
\end{align}

\paragraph{Notation for error terms.}
We write $f(\sigma)=O_M(g(\sigma))$ as $\sigma\to\infty$ if there exists a 
constant 
$c(M)>0$ depending on the parameter $M$ such that 
$|f(\sigma)|\le c(M)|g(\sigma)|$ for all $\sigma$ large enough. 
If we use $O$ with more than one subscript, the constant may depend on all
subscripts.

The following theorem connects the density $\rho_{i_0}^\single$
defined in \eqref{eq:def-rho-single} to the asymptotics of VRJP. 
Recall that for $\vec T\in\vec\T_{i_1}$, its undirected version  
is denoted by $T\in\T$.

\begin{theorem}{\bf (Limiting joint density)}
\label{thm:claim-thm-asymptotics-collected-single}
Let $i_0,i_1\in V$, $k\in\K_{i_0,i_1}^+$, $\sigma>0$, 
$A\subseteq\cL_\sigma$ be measurable, and $\vec T\in\vec\T_{i_1}$.
For $M>0$, on the events $B_\sigma(M)$, one has the following 
in the limit as $\sigma\to\infty$ with $l_{i_0}=\sigma/\sum_{i\in V}e^{2v_i}$:
\begin{align}
 & P_{i_0}(k(\sigma)=k,l(\sigma)\in A,\vec T^\lastexit(0,\sigma)=\vec T)\nonumber\\
= &\left(1+O_{M,W,G}\left(\sigma^{-1/2}\right)\right)
\int_A l_{i_0}^{\frac{|V|-1}{2}-|E|}\rho^\single_{i_0}(\kappa,v,i_1,T) \prod_{i\in V\setminus\{i_0\}}dv_i
\label{eq:claim-thm-asymptotics-collected-single}
\end{align}
\end{theorem}

This theorem is a main ingredient to prove the 
following weak convergence result. Let 
\begin{align}
\label{eq:def-xi-single}
\xi_\sigma= & (k(\sigma),l(\sigma),Z_\sigma,T^\lastexit(0,\sigma)), \\
\label{eq:def-event-Q-i0}
\cQ_{\sigma,i_0}= & \bigcup_{i_1\in V}\K_{i_0,i_1}\times\cL_\sigma\times\{i_1\}\times\T_{i_1}. 
\end{align}
Thus, $\xi_\sigma\in\cQ_{\sigma,i_0}$ means that all $l_i$ are positive and 
$T^\lastexit(0,\sigma)$ is a spanning tree.

\begin{theorem}{\bf (Weak convergence)}
\label{cor:weak-convergence-single-time}
The joint sub-proba\-bi\-li\-ty distribution of
\begin{align}
& (\kappa(\sigma),v(\sigma),Z_\sigma,T^\lastexit(0,\sigma))
\end{align}
with respect to $P_{i_0}(\cdot\cap\{\xi_\sigma\in\cQ_{\sigma,i_0}\})$ converges 
weakly as $\sigma\to\infty$ to the probability measure 
$\mu^\single_{i_0}$ defined in \eqref{eq:def-mu-single}. 
\end{theorem}

\subsection{Double timescale -- extended $H^{2|2}$ model as limit of VRJP}

In this section, we generalize the results from 
Section \ref{sec:single-time-scale} to two timescales $\sigma,\sigma'>0$
to retrieve the $\htwo$ model as a marginal of $\mu_{i_0}^\bigsusy$. 
First, let us introduce double timescale versions of the quantities and 
sets considered in Section \ref{sec:single-time-scale}. 

We set $l'(\sigma,\sigma')=(l_i'(\sigma,\sigma'))_{i\in V}$ with 
\begin{align}
l_i'(\sigma,\sigma')=l_i(\sigma+\sigma')-l_i(\sigma).
\end{align}
This is the local time the process $(Z_\sigma)_{\sigma\ge 0}$ spends in vertex
$i$ during the time interval $[\sigma,\sigma+\sigma']$. Using the 
definition \eqref{eq:def-L} of $\cL_\sigma$, we set 
$\cL_{\sigma,\sigma'}=\cL_\sigma\times\cL_{\sigma'}$.

For the time interval of length $\sigma'$ starting at $\sigma$, 
define $k'(\sigma,\sigma')=(k_{ij}'(\sigma,\sigma'))_{(i,j)\in\vec{E}}$ by 
\begin{align}
k_{ij}'(\sigma,\sigma')=k_{ij}(\sigma+\sigma')-k_{ij}(\sigma). 
\end{align}
In other words, $k_{ij}'(\sigma,\sigma')$ equals the number of crossings 
from $i$ to $j$ in the time interval $[\sigma,\sigma+\sigma']$. 
For given $i_0,i_1,i_1'\in V$, let 
\begin{align}
\label{eq:def-K}
\K_{i_0,i_1,i_1'}=\K_{i_0,i_1}\times\K_{i_1,i_1'}\quad\text{and}\quad
\K_{i_0,i_1,i_1'}^+=\K_{i_0,i_1,i_1'}\cap(\N^{\vec E}\times\N^{\vec E}). 
\end{align}
Thus, $\K_{i_0,i_1,i_1'}^+$ is obtained by a restriction to strictly positive integers. 
Note that one has $(k(\sigma),k'(\sigma,\sigma'))\in\K_{i_0,Z_\sigma,Z_{\sigma+\sigma'}}$. 

The following definition introduces some events which are useful 
to study the joint law of the random variable 
\begin{align}
\label{eq:def-xi}
\xi_{\sigma,\sigma'}=(k(\sigma),k'(\sigma,\sigma'),l(\sigma),l'(\sigma,\sigma'),
Z_\sigma,Z_{\sigma+\sigma'},T^\lastexit(0,\sigma),T^\lastexit(\sigma,\sigma+\sigma')).
\end{align}

\begin{definition}[Events concerning local times and last exit trees]
\label{def:events}
Let $i_0,i_1,i_1'\in V$, $(k,k')\in\K_{i_0,i_1,i_1'}^+$, $\sigma,\sigma'>0$, 
$A\subseteq\cL_{\sigma,\sigma'}$ be measurable, $\vec T\in\vec\T_{i_1}$, and 
$\vec T'\in\vec\T_{i_1'}$. 
In this setup, we define the following events
\begin{align}
K_{k,\sigma,k',\sigma'}= & \{ k(\sigma)=k,k'(\sigma,\sigma')=k'\}, \\
\label{eq:def-event-L}
L_{\sigma,\sigma'}(A)= & \{ (l(\sigma),l'(\sigma,\sigma')) \in A\}, \\
E_{i_1,\vec T,\sigma,i_1',\vec T',\sigma'}=&
\{\vec T^\lastexit(0,\sigma)=\vec T,\vec T^\lastexit(\sigma,\sigma+\sigma')=\vec T'\}.
\end{align}
\end{definition}
The following theorem describes explicitly the distribution of the random variable
$\xi_{\sigma,\sigma'}$. 

\begin{theorem}{\bf (Joint density of edge crossings, local times, and 
last exit trees)}
\label{thm:proba-of-a-path}
In the setup of Definition \ref{def:events}, the following holds
with $\pp$ defined in \eqref{eq:def-pp}
\begin{align}
&P_{i_0}(K_{k,\sigma,k',\sigma'}\cap L_{\sigma,\sigma'}(A)\cap E_{i_1,\vec T,\sigma,i_1',\vec T',\sigma'})\cr
= & \int_A
\exp\Big(
\sum_{\{i,j\}\in E} W_{ij}\left(1-\sqrt{1+l_i+l_i'}\sqrt{1+l_j+l_j'}\right)\Big)
\prod_{i\in V\setminus\{i_1'\}}\frac{1}{\sqrt{1+l_i+l_i'}} \cr
& \cdot \pp(k,l,\vec T) \pp(k',l',\vec T')
\prod_{i\in V\setminus\{i_0\}}dl_idl_i'.
\label{eq:density-path}
\end{align}
\end{theorem}

Recall that the motivation for taking two different time scales 
$1\ll\sigma\ll\sigma'$ was to study fluctuations of local times. 
In that view, it is natural to take the limit $\sigma'\to\infty$
first and only second the limit $\sigma\to\infty$. More generally, it turns out
that we can also take $\sigma$ and $\sigma'$ simultaneously to infinity, 
as long as $\min\{\sigma,\sigma'\sigma^{-2}\}\to\infty$. 
We will abbreviate henceforth
\begin{eqnarray}
k'=k(\sigma,\sigma')\quad\text{ and }\quad l'=l(\sigma,\sigma').
\end{eqnarray}

\paragraph{Rescaling of local times and their fluctuations.}
The considerations around \eqref{eq:u-as-limit-for-vrjp} motivate us to study
the cross-ratio 
\begin{align}
\label{eq:cross-ratio}
\frac{l_i/l_{i_0}}{l_i'/l_{i_0}'}. 
\end{align}
In order not to divide by $0$, given $\sigma,\sigma'>0$, we consider the event 
$\{l_i>0,\; l_i'>0\text{ for all }i\in V\}=\{(l,l')\in\cL_{\sigma,\sigma'}\}$.
On this event, in analogy to $v_i$ defined in \eqref{eq:def-v}, 
we introduce new variables $u_i=u_i(\sigma,\sigma')$ for $i\in V$ by 
\begin{align}
\label{eq:def-u-v}
l_i'=l_{i_0}' e^{2u_i}. 
\end{align}
Although $v_i$ and $u_i$ are certainly different random variables, formula
\eqref{eq:u-as-limit-for-vrjp} shows us that they coincide $P_{i_0}$-almost
surely asymptotically in the limit as $\sigma'\gg\sigma\to\infty$. 
Not very unexpectedly for fluctuations, the right scale for the logarithm 
of the cross-ratio \eqref{eq:cross-ratio} turns out to be $\sqrt{l_{i_0}}$, 
i.e.\ roughly the square root of the smaller time scale. This motivates
us to define 
\begin{align}
\label{eq:def-s}
s_i=s_i(\sigma,\sigma')
=-\frac12\sqrt{l_{i_0}}\log\frac{l_i/l_{i_0}}{l_i'/l_{i_0}'}
=\sqrt{l_{i_0}}(u_i-v_i), \quad i\in V.
\end{align}
Note that $s_{i_0}=u_{i_0}=0$. Hence, $(s_i)_{i\in V},(u_i)_{i\in V}\in\Omega_{i_0}$.
An interpretation of $s_i$ is most easily described in the special case of 
taking the limit $\sigma'\to\infty$ first and only then $\sigma\to\infty$. 
In this case, $\lim_{\sigma'\to\infty}s_i(\sigma,\sigma')$ describes the 
fluctuations of $l_i(\sigma)/l_{i_0}(\sigma)$ around 
$\lim_{\sigma'\to\infty}l_i(\sigma')/l_{i_0}(\sigma')$ on the appropriate scale. 

Analogously to \eqref{eq:def-kappa}, we introduce, again on the event 
$\{(l,l')\in\cL_{\sigma,\sigma'}\}$, the rescaled 
crossing numbers 
$\kappa_{ij}'=\kappa_{ij}'(\sigma,\sigma')$ for $(i,j)\in\vec E$ by 
\begin{align}
& \kappa_{ij}'=\frac{k_{ij}'-\frac12W_{ij}\sqrt{l_i'l_j'}}{\sqrt{l_{i_0}'}}
=\frac{k_{ij}'}{\sqrt{l_{i_0}'}}-\frac{W_{ij}}{2}e^{u_i+u_j}{\sqrt{l_{i_0}'}}.
\label{eq:def-kappa-prime}
\end{align}

For a truncation parameter $M>0$, we consider the events 
\begin{align}
\label{eq:def-event-B}
& B_{\sigma,\sigma'}(M)=\{ |\kappa_{ij}|,|\kappa_{ij}'|,|s_i|,|u_i|,|v_i|\le M
\text{ and }l_i>0,l_i'>0
\text{ for all }i,j\in V\}.
\end{align}
The random variable $\xi_{\sigma,\sigma'}$ defined in \eqref{eq:def-xi}
is only interesting on the event $\{\xi_{\sigma,\sigma'}\in\cO_{\sigma,\sigma',i_0}\}$ with 
\begin{align}
\label{eq:def-event-O-i0}
\cO_{\sigma,\sigma',i_0}=\bigcup_{i_1,i_1'\in V}\K_{i_0,i_1,i_1'}
\times\cL_{\sigma,\sigma'}\times\{i_1\}\times\{i_1'\}\times\T_{i_1}\times\T_{i_1'}. 
\end{align}
One has $\xi_{\sigma,\sigma'}\not\in\cO_{\sigma,\sigma',i_0}$ if some $l_i$ or $l_i'$ equals $0$ or if 
$T^\lastexit(0,\sigma)$ or $T^\lastexit(\sigma,\sigma+\sigma')$ is not spanning. 
Furthermore, we consider the map 
\begin{align}
\label{eq:def-F-sigma-sigma'-i0}
& F_{\sigma,\sigma',i_0}:\cO_{\sigma,\sigma',i_0}\to (\R^{\vec E})^2\times\Omega_{i_0}^3\times V^2\times\T^2  ,\cr
& F_{\sigma,\sigma',i_0}(k,k',l,l',i_1,i_1',T,T')
= (\kappa,\kappa',s,v,u,i_1,i_1',T,T')
\end{align}
defined by the equations  \eqref{eq:def-v}, \eqref{eq:def-kappa}, and 
\eqref{eq:def-u-v}--\eqref{eq:def-kappa-prime}. 

The following main theorem shows that the extended $\htwo$ model
describes the asymptotics of the time-changed version $Z$ of the 
vertex-reinforced jump process. To be more precise, 
it occurs as the joint limit of the rescaled crossing numbers, the rescaled 
fluctuations of local times, the asymptotic local times on a 
logarithmic scale, the endpoints of paths, and last exit trees as follows. 
Let $E_{i_0}$ denote the expectation with respect to $P_{i_0}$.

\begin{theorem}{\bf (Weak convergence to the extended $\htwo$ model)}
\label{thm:weak-convergence}
The joint sub-probability distribution of
\begin{align}
\label{eq:vector-kappa-etc}
& (\kappa(\sigma),\kappa'(\sigma,\sigma'),s(\sigma,\sigma'),v(\sigma), 
u(\sigma,\sigma'),
Z_\sigma,Z_{\sigma+\sigma'},T^\lastexit(0,\sigma),T^\lastexit(\sigma,\sigma+\sigma'))
\end{align}
with respect to $P_{i_0}(\cdot\cap\{\xi_{\sigma,\sigma'}\in\cO_{\sigma,\sigma',i_0}\})$ converges 
weakly as $\min\{\sigma,\sigma'\sigma^{-2}\}\to\infty$ to $\mu^\bigsusy_{i_0}$. 
In other words, for any bounded continuous test function 
$f:(\R^{\vec E})^2\times\Omega_{i_0}^3\times V^2\times\T^2\to\R$, one has 
\begin{align}
\lim_{\min\{\sigma,\sigma'\sigma^{-2}\}\to\infty}
E_{i_0}\left[f(F_{\sigma,\sigma',i_0}(\xi_{\sigma,\sigma'})),\xi_{\sigma,\sigma'}\in\cO_{\sigma,\sigma',i_0}\right]
=\int\limits_{\cH^2\times\Omega_{i_0}^3\times V^2\times\T^2}
f\, d\mu^\bigsusy_{i_0}.
\end{align}
In particular, $P_{i_0}(\xi_{\sigma,\sigma'}\in\cO_{\sigma,\sigma',i_0})\to 1$ 
as $\min\{\sigma,\sigma'\sigma^{-2}\}\to\infty$.
\end{theorem}

\paragraph{How this paper is organized.}
In Section~\ref{sec:proof-single} we prove the single timescale results 
stated in Section~\ref{sec:single-time-scale}, which could in fact be deduced
from the double timescale analysis in Section~\ref{se:double-timescale}.  
However, for the convenience of the reader, we present first the easier argument
for a single timescale before we prove the more general results for two timescales.
In Section~\ref{sec:21}, 
we prove Theorem \ref{thm:proba-of-a-path-single} on the density of edge 
crossings, local times, and last exit tree, using path counting arguments and 
calculating volume factors. In Section~\ref{sec:22}, 
we state Lemma \ref{le:first-asymptotics-single} which gives the asymptotics of 
$\pp(k,l,\vec T)$. It is proved in Appendix~\ref{appendix:combinatorial-factors}. 
From this we deduce Lemma \ref{le:second-asymptotics-single}
giving the asymptotic density of a path. In Section~\ref{sec:23} we prove Theorem 
\ref{thm:claim-thm-asymptotics-collected-single} giving the limiting joint density
of currents, local times, and last exit tree. Section~\ref{sec:24}
proves the \emph{vague} convergence  of 
$(\kappa(\sigma),v(\sigma),Z_\sigma,T^\lastexit(0,\sigma))$
rather than \emph{weak} convergence. This involves a continuum limit. 
Finally, in Section~\ref{sec:25} we prove the 
weak convergence theorem \ref{cor:weak-convergence-single-time}.
This requires a Gaussian integral over currents on the graph which is 
stated in Lemma \ref{le:integral-kappa-single} of that section and proved in 
appendix \ref{app:B}. Another key ingredient for the weak convergence result
is the normalization \eqref{eq:normalizing-const} of 
$\mu_{i_0}^\susy$. It is reviewed in Appendix~\ref{sec:app-review-results}. 

Section~\ref{se:double-timescale} deals with the double timescale results. 
Section~\ref{se:joint-density-local-times} contains a proof of 
Theorem \ref{thm:proba-of-a-path} giving the density of the random 
variable $\xi_{\sigma,\sigma'}$ for fixed times $\sigma,\sigma'$. It is in the spirit 
of the proof of Theorem \ref{thm:proba-of-a-path-single}.
In Section~\ref{se:asymptotics}, we derive the asymptotics of this 
density, appropriately rescaled, in the limit 
$\min\{\sigma,\sigma'\sigma^{-2}\}\to\infty$. This yields a proof of 
Theorem \ref{thm:claim-thm-asymptotics-collected} which states a double 
timescale version similar to Theorem 
\ref{thm:claim-thm-asymptotics-collected-single}.
In Section~\ref{se:continuum-limit} we show the two timescale variant of 
vague convergence. In Section~\ref{se:marginals}, we deduce Theorem 
\ref{thm:weak-convergence} giving weak convergence. The key ingredients here 
are on the one hand again the normalization \eqref{eq:normalizing-const} and 
on the other hand the fact, stated in Theorem 
\ref{thm:marginal-rho-susy}, that the extended $\htwo$ 
model $\mu_{i_0}^\bigsusy$ has $\mu_{i_0}^\susy$ as a marginal.

\section{Proof for a single timescale}
\label{sec:proof-single}

\subsection{Proof of Theorem \ref{thm:proba-of-a-path-single}}
\label{sec:21}

For $0<\sigma_1<\sigma_2$, let $\discrete(Z_{[\sigma_1,\sigma_2]})$ 
denote the path in discrete time obtained from $(Z_\sigma)_{\sigma\in[\sigma_1,\sigma_2]}$
by taking only the values immediately before the jumps. 
For $i_0,i_1\in V$, $k\in\K^+_{i_0,i_1}$, and $\vec T\in\vec\T_{i_1}$ let 
$\Pi_{i_0,i_1}(k,\vec T)$ denote the set of finite paths in discrete time which 
start in $i_0$, end in $i_1$, cross every $(i,j)\in\vec E$ precisely $k_{ij}$
times and have last exit tree $\vec T$. 

Let $\sigma$ be fixed. We derive the joint density of 
$(k(\sigma),l(\sigma),\vec T^\lastexit(0,\sigma))$ using combinatorial arguments.
The density of paths for VRJP was first provided in the proof of Theorem 3 
in \cite{sabot-tarres2016}. As it differs slightly in notation and time scaling
from the present paper, we explain the connection to the following formula 
\eqref{eq:prob-density-fixed-path-single} in Appendix~\ref{sec:app-review-results}, 
cf.\ formula \eqref{eq:density-s-t-2}. 
For any path $\pi\in\Pi_{i_0,i_1}(k,\vec T)$ and any 
measurable $A\subseteq\cL_\sigma$ we obtain with an appropriate volume factor 
$\V(k,l,i_1)$ specified in \eqref{eq:volume-factor-k-single} below: 
\begin{align}
\label{eq:prob-density-fixed-path-single}
& P_{i_0}(\discrete(Z_{[0,\sigma]})=\pi,l(\sigma)\in A)
= \int_A\exp\Big(
\sum_{\{i,j\}\in E} W_{ij}\Big(1-\sqrt{1+l_i}\sqrt{1+l_j}\Big)\Big)\cr
& \cdot \prod_{i\in V\setminus\{i_1\}}\frac{1}{\sqrt{1+l_i}} 
\prod_{(i,j)\in\vec{E}}\left(\frac{W_{ij}}{2}\right)^{k_{ij}}
\cdot\V(k,l,i_1)\prod_{i\in V\setminus\{i_0\}} dl_i. 
\end{align}
Note that the right hand side in the last equation depends only on the 
choice of $k$, $i_0$, and $i_1$, but neither on $\vec T$ nor 
on the choice of $\pi\in\Pi_{i_0,i_1}(k,\vec T)$. Consequently, 
\begin{align}
& P_{i_0}(k(\sigma)=k,l(\sigma)\in A,\vec T^\lastexit(0,\sigma)=\vec T)\cr
= & |\Pi_{i_0,i_1}(k,\vec T)|
\int_A\exp\Big(
\sum_{\{i,j\}\in E} W_{ij}\Big(1-\sqrt{1+l_i}\sqrt{1+l_j}\Big)\Big)\cr
& \cdot 
\prod_{i\in V\setminus\{i_1\}}\frac{1}{\sqrt{1+l_i}} 
\prod_{(i,j)\in\vec{E}}\left(\frac{W_{ij}}{2}\right)^{k_{ij}}
\cdot\V(k,l,i_1)\prod_{i\in V\setminus\{i_0\}} dl_i.
\label{eq:proba-under-consideration-single}
\end{align}
The volume factor $\V(k,l,i_1)$ consists of a product
of contributions from each vertex. We determine it as follows. 
For $i\in V$, we let 
\begin{align}
\label{eq:def-k-i}
k_i=\sum_{\substack{j\in V:\\ \{i,j\}\in E}}k_{ij}
\end{align}
be the number of departures from vertex $i$. 
Given the directed edge crossings $k$, to have for all vertices $i\in V$ local 
time $l_i$ at vertex $i$ at time $\sigma$ we need jump times 
$0=t_0^{(i)}<t_1^{(i)}<\cdots<t_{k_i}^{(i)}\le l_i$, where for all $i\in V\setminus\{i_1\}$ 
we have moreover $t_{k_i}^{(i)}=l_i$. For $i\neq i_1$, these are $k_i-1$ jumps in the time 
interval $(0,l_i)$. Integrating over $t_1^{(i)},\ldots,t_{k_i-1}^{(i)}$ gives the volume
factor contribution from vertex $i\neq i_1$ 
\begin{align}
\cV(k_i-1,l_i) = 
\lambda^{k_i-1}\big(\big\{
(t_1^{(i)},\ldots,t_{k_i-1}^{(i)})\in(0,l_i)^{k_i-1}:\;t_1^{(i)}<\cdots<t_{k_i-1}^{(i)}
\big\}\big)
=\frac{l_i^{k_i-1}}{(k_i-1)!}; 
\end{align}
here $\lambda^{k_i-1}$ denotes the Lebesgue measure on $\R^{k_i-1}$. 
For $i=i_1$, given $l_{i_1}$, there is one degree of freedom more. Integrating over 
the jump times $t_1^{(i_1)},\ldots,t_{k_{i_1}}^{(i_1)}$ gives the volume
factor 
\begin{align}
\cV(k_{i_1},l_{i_1})=
\lambda^{k_{i_1}}\big(\big\{
(t_1^{(i_1)},\ldots,t_{k_{i_1}}^{(i_1)})\in(0,l_{i_1})^{k_{i_1}}:\;
t_1^{(i_1)}<\cdots<t_{k_{i_1}}^{(i_1)}\big\}\big)
=\frac{l_{i_1}^{k_{i_1}}}{k_{i_1}!}.
\label{eq:Vk}
\end{align}
Altogether this yields the volume factor 
\begin{align}
\label{eq:volume-factor-k-single}
\V(k,l,i_1)=\cV(k_{i_1},l_{i_1})\prod_{i\in V\setminus\{i_1\}}\cV(k_i-1,l_i)=
\frac{l_{i_1}^{k_{i_1}}}{k_{i_1}!}\prod_{i\in V\setminus\{i_1\}} \frac{l_i^{k_i-1}}{(k_i-1)!}. 
\end{align}
The cardinality of $\Pi_{i_0,i_1}(k,\vec T)$ is counted in \cite{Keane-Rolles2000};
it is reviewed in Fact \ref{eq:fact-counting-paths} in the appendix. 
Combining this with 
\eqref{eq:volume-factor-k-single} and using that $\vec T$ is a spanning tree 
directed towards $i_1$, we obtain 
\begin{align}
\label{eq:factor-kl-single}
|\Pi_{i_0,i_1}(k,\vec T)| \V(k,l,i_1)
=\frac{\prod_{(i,j)\in \vec T}k_{ij}}{\prod_{(i,j)\in\vec{E}}k_{ij}!}
\frac{\prod_{i\in V}l_i^{k_i}}{\prod_{i\in V\setminus\{i_1\}} l_i}
=\prod_{(i,j)\in\vec{E}}\frac{l_i^{k_{ij}}}{k_{ij}!}
\prod_{(i,j)\in \vec T}\frac{k_{ij}}{l_i}.
\end{align}
The claim follows from \eqref{eq:proba-under-consideration-single} and 
\eqref{eq:factor-kl-single}.

\subsection{Asymptotics of the density of a path}
\label{sec:22}

In this section, we use Taylor arguments and Stirling's formula to 
asymptotically describe the density of the random variables
$(k(\sigma),l(\sigma),\vec T^\lastexit(0,\sigma))$ 
but rewritten in terms of rescaled variables. 
Recall the definition \eqref{eq:def-omega} of $\omega_{ij}$.

\paragraph{Scales of the variables.} 
Recall from \eqref{eq:def-kappa} and \eqref{eq:def-v} the following
relations:
\begin{align}
\label{eq:kij-in-terms-of-omega-single}
k_{ij}=l_{i_0}\omega_{ij}+\sqrt{l_{i_0}}\kappa_{ij}, \qquad
\frac{\sigma}{l_{i_0}}=\sum_{i\in V} e^{2v_i}.
\end{align} 
In particular, $\sigma$ and $l_{i_0}$ live on the same scale when all $v_i$ are  
bounded. Thus, on the event $B_\sigma(M)$ defined in 
\eqref{eq:def-event-B-single}, 
all $l_i$ and all $k_{ij}$ have the same order of magnitude as~$\sigma$.

\begin{lemma}{\bf (Asymptotics of the combinatorial factors)}
\label{le:first-asymptotics-single}
Given $M>0$, on the event $B_\sigma(M)$, one has the following 
asymptotics 
for $\pp$ defined in \eqref{eq:def-pp} as $\sigma\to\infty$:
\begin{align}
\pp(k,l,\vec T)
= & \frac{1}{(2\pi l_{i_0})^{|E|}} 
\exp\Big(l_{i_0}\sum_{(i,j)\in\vec E}\omega_{ij}-v_{i_1}
-\sum_{(i,j)\in\vec E}\frac{\kappa_{ij}^2}{2\omega_{ij}}\Big) \nonumber\\
& \cdot\prod_{i\in V\setminus\{i_1\}}e^{-2v_i}
\prod_{(i,j)\in\vec{E}\setminus\vec T}\frac{1}{\omega_{ij}}
\prod_{\{i,j\}\in E}\omega_{ij}\cdot(1+O_{M,W,G}(\sigma^{-1/2})).
\label{eq:lemma-prod-Wl-hoch-k-single}
\end{align}
\end{lemma}

The proof relies on Taylor expansions and Stirling's formula.
It is given in appendix~\ref{appendix:combinatorial-factors}.

Next, we study the asymptotic behavior of the remaining part of the density 
given in \eqref{eq:density-path-single}.

\begin{lemma}{\bf (Asymptotics of the density of a path)}
\label{le:second-asymptotics-single}
For $M>0$, on the events $B_\sigma(M)$, one has the following 
in the limit as $\sigma\to\infty$:
\begin{align}
& \exp\Big(
\sum_{\{i,j\}\in E} W_{ij}\left(1-\sqrt{1+l_i}\sqrt{1+l_j}\right)\Big)
\prod_{i\in V\setminus\{i_1\}}\frac{1}{\sqrt{1+l_i}} 
= \left(1+O_{M,W,G}\left(\sigma^{-1}\right)\right)\nonumber\\
& \cdot (l_{i_0})^{-\frac{|V|-1}{2}} \exp\Big(\sum_{\{i,j\}\in E} W_{ij}\left(
1-\cosh(v_i-v_j)\right)\Big) 
\prod_{(i,j)\in\vec{E}}\exp(-l_{i_0}\omega_{ij})
\prod_{i\in V\setminus\{i_1\}}e^{-v_i} .
\label{eq:2nd-asymptotics-single}
\end{align}
\end{lemma}
\begin{proof}
During the proof, we work on the events $B_\sigma(M)$ in the limit
$\sigma\to\infty$. Using the Taylor expansion of $(1+l_i)^{1/2}$
and the representation \eqref{eq:def-v} of $l_i,l_j$, we obtain  
\begin{align}
\sqrt{1+l_i}\sqrt{1+l_j}
= & \sqrt{l_il_j}+\frac12\left(\sqrt{\frac{l_i}{l_j}}
+\sqrt{\frac{l_j}{l_i}}\right)+O_M\left(\sigma^{-1}\right)
\nonumber\\
=&\sqrt{l_il_j}+\cosh(v_i-v_j)+O_M\left(\sigma^{-1}\right), 
\label{eq:prod-sqrt-l}\\
\label{eq:prod-omega}
\prod_{(i,j)\in\vec{E}}\exp(l_{i_0}\omega_{ij}) 
= &\prod_{\{i,j\}\in E}\exp(2l_{i_0}\omega_{ij}) 
= \exp\Big(\sum_{\{i,j\}\in E}W_{ij}\sqrt{l_il_j}\Big). 
\end{align}
Consequently, we deduce 
\begin{align}
& \exp\Big(
-\sum_{\{i,j\}\in E} W_{ij}\sqrt{1+l_i}\sqrt{1+l_j}\Big)
\prod_{(i,j)\in\vec{E}}\exp(l_{i_0}\omega_{ij}) \nonumber\\
= & \exp\Big(-\sum_{\{i,j\}\in E} W_{ij}\cosh(v_i-v_j)\Big)
\left(1+O_{M,W,G}\left(\sigma^{-1}\right)\!\right).
\label{eq:exp-sqrt-l-l2}
\end{align}
We conclude by the observation that 
\begin{align}
\prod_{i\in V\setminus\{i_1\}}\frac{1}{\sqrt{1+l_i}} 
=\prod_{i\in V\setminus\{i_1\}}\frac{1}{\sqrt{l_i}(1+O_M(\sigma^{-1}))} 
= &  \frac{1+O_{M,G}(\sigma^{-1})}{l_{i_0}^{\frac{|V|-1}{2}}}
\prod_{i\in V\setminus\{i_1\}}e^{-v_i}.
\end{align}
\end{proof}

\subsection{Proof of Theorem \ref{thm:claim-thm-asymptotics-collected-single}}
\label{sec:23}

Combining Theorem \ref{thm:proba-of-a-path-single}, Lemmas 
\ref{le:first-asymptotics-single} and \ref{le:second-asymptotics-single}, 
we obtain 
\begin{align}
& P_{i_0}(k(\sigma)=k,l(\sigma)\in A,\vec T^\lastexit(0,\sigma)=\vec T)
\nonumber\\
=& (2\pi)^{-|E|}\int_A l_{i_0}^{-|E|-\frac{|V|-1}{2}}
\exp\Big(\sum_{\{i,j\}\in E} W_{ij}\left(1-\cosh(v_i-v_j)\right)
-v_{i_1}
-\sum_{(i,j)\in\vec E}\frac{\kappa_{ij}^2}{2\omega_{ij}}\Big) \nonumber\\
& \cdot\prod_{i\in V\setminus\{i_1\}}e^{-3v_i}
\prod_{(i,j)\in\vec{E}\setminus\vec T}\frac{1}{\omega_{ij}}
\prod_{\{i,j\}\in E}\omega_{ij}\cdot(1+O_{M,W,G}(\sigma^{-1/2}))\,
\prod_{i\in V\setminus\{i_0\}} dl_i. 
\label{eq:proba-k-l-vec-T}
\end{align}
The symmetry $\omega_{ij}=\omega_{ji}$ yields the following 
formula, which connects products indexed by directed edges with products 
indexed by undirected edges:
\begin{align}
& \prod_{(i,j)\in\vec{E}\setminus\vec T}\frac{1}{\omega_{ij}}
\prod_{\{i,j\}\in E}\omega_{ij}
= \prod_{\{i,j\}\in T}\omega_{ij}\prod_{\{i,j\}\in E}\frac{1}{\omega_{ij}}
=2^{|E|-|V|+1}\prod_{\{i,j\}\in E\setminus T}\frac{1}{W_{ij}e^{v_i+v_j}}.
\label{eq:prod-trees-single}
\end{align}
Therefore, using the definition \eqref{eq:def-rho-single} of $\rho_{i_0}^\single$, 
the right hand side of \eqref{eq:proba-k-l-vec-T} is given by 
\begin{align}
2^{1-|V|}\int_A l_{i_0}^{-|E|-\frac{|V|-1}{2}} \sum_{j\in V}e^{2v_j}\prod_{i\in V}e^{-2v_i}
\cdot(1+O_{M,W,G}(\sigma^{-1/2}))\rho_{i_0}^\single\prod_{i\in V\setminus\{i_0\}} dl_i.
\end{align} 
We conclude the proof by transforming $l$-variables to $v$-variables.
Using $v_{i_0}=0$, 
\begin{align}
\frac{\partial v_i}{\partial l_j}
=\frac{1}{2l_i}\left(\delta_{ij}+\frac{l_i}{l_{i_0}}\right),\quad
i,j\in V\setminus\{i_0\}, 
\label{eq:dv-dl}
\end{align}
and $\det(\Id+A)=1+\operatorname{tr}(A)$ for rank 1 matrices $A$, 
we obtain the Jacobi determinant 
\begin{align}
\left|\det \frac{\partial v}{\partial l}\right|
=&\Big(1+\sum_{i\in V\setminus\{i_0\}}\frac{l_i}{l_{i_0}}\Big)
\prod_{i\in V\setminus\{i_0\}}\frac{1}{2l_i}
=\frac{\sigma}{l_{i_0}}\prod_{i\in V\setminus\{i_0\}}\frac{1}{2l_i}
= (2l_{i_0})^{1-|V|}\sum_{j\in V}e^{2v_j}\prod_{i\in V}e^{-2v_i} . 
\label{eq:det-dv-dl}
\end{align}

\subsection{Proof of vague convergence}
\label{sec:24}

\begin{lemma}
\label{le:vague-conv-single}
For any continuous compactly supported function 
$f:\R^{\vec E}\times\Omega_{i_0}\times V\times\T\to\R$, one has 
\begin{align}
\label{eq:vague-singletime}
\lim_{\sigma\to\infty}
E_{i_0}\left[f(\kappa(\sigma),v(\sigma),Z_\sigma,T^\lastexit(0,\sigma)),
\xi_\sigma\in\cQ_{\sigma,i_0}\right]
=\int\limits_{\cH\times\Omega_{i_0}\times V\times\T}
f\, d\mu^\single_{i_0}.
\end{align}
\end{lemma}
\begin{proof}
Recall that $l_i=l_i(v,\sigma)=\sigma e^{2v_i}/\sum_{j\in V}e^{2v_j}$. Hence, 
$\kappa=\kappa(k,v,\sigma)$ is a function of $k$, $v$, and $\sigma$, cf.\ 
\eqref{eq:def-kappa}.
Because $f$ is compactly supported, we can choose a constant $M>0$ such 
that for any $k\in\K_{i_0,i_1}$ and $l\in\cL_\sigma$ with 
$(\kappa,v,i_1,T)\in\supp f$, 
all the components of $\kappa(k,l)$ and $v(l)$ are bounded in absolute value
by $M$. Moreover, all components 
of $l/\sigma$ are bounded away from $0$, and bounded above by $1$. 
For $\{i,j\}\in E$, the facts 
$\omega_{ij}\ge\frac{W_{ij}}{2}e^{-2M}>0$ and 
$|\kappa_{ij}|\le M$ together with \eqref{eq:kij-in-terms-of-omega-single} 
imply $k_{ij}>0$ for $\sigma$ large enough, and thus $k\in\K_{i_0,i_1}^+$. 
Note that $\vec T^\lastexit(0,\sigma)=\vec T\in\vec\T_{i_1}$ is equivalent 
to $T^\lastexit(0,\sigma)=T$ and $Z_\sigma=i_1$. 
Hence, Theorem \ref{thm:claim-thm-asymptotics-collected-single} yields 
that the l.h.s.\ of \eqref{eq:vague-singletime} in the limit as 
$\sigma\to\infty$ equals 
\begin{align}
\label{eq:first-riemann-sum}
& \left(1+O_{M,W,G}\left(\sigma^{-1/2}\right)\right) 
\sum_{\substack{i_1\in V,\\ T\in\T}}\int_{\Omega_{i_0}}
l_{i_0}^{\frac{|V|-1}{2}-|E|}\sum_{k\in\K_{i_0,i_1}}
(f\rho_{i_0}^\single)(\kappa,v,i_1,T) \prod_{i\in V\setminus\{i_0\}} dv_i. 
\end{align}
Fix a path from $i_0$ to $i_1$. Let $\pi=(\pi_{ij})_{(i,j)\in\vec E}$
be the corresponding edge crossing numbers. 
We introduce the shift vector 
$\vartheta_l=(l_{i_0}^{1/2}\omega_{ij})_{(i,j)\in\vec E}\in\cH$. Let 
$\vec T_0$ be an arbitrary directed spanning tree of $G$. 
Let $\Gamma\subset\cH$ denote the lattice which has $\Z^{\vec E\setminus\vec T_0}$
as its image under the restriction map $\R^{\vec E}\to\R^{\vec E\setminus\vec T_0}$.
When $k$ runs over $\K_{i_0,i_1}$, the corresponding 
$\kappa-l_{i_0}^{-1/2}\pi$ runs over the shifted lattice
$l_{i_0}^{-1/2}\Gamma-\vartheta_l$.
In other words, for any $v\in\Omega_{i_0}$ and $T\in\T$, one has 
\begin{align}
\label{eq:riemann-sum-f-circ-F-single}
& \sum_{k\in\K_{i_0,i_1}}(f\rho_{i_0}^\single)(\kappa(k,v),v,i_1,T)
=\sum_{\tilde\kappa\in l_{i_0}^{-1/2}\Gamma-\vartheta_l}
(f\rho_{i_0}^\single)\Big(\tilde\kappa+\frac{\pi}{\sqrt{l_{i_0}}},
v,i_1,T\Big),
\end{align}
Multiplying this with 
$l_{i_0}^{\frac{|V|-1}{2}-|E|}=l_{i_0}^{-\frac{|\vec E\setminus\vec T_0|}{2}}=l_{i_0}^{-\dim\cH}$,
we interpret it as a Riemann sum. It converges to an integral, uniformly 
for $v$ in compact sets. We conclude that the expression in 
\eqref{eq:first-riemann-sum} converges as $\sigma\to\infty$ to the r.h.s.\ 
in formula \eqref{eq:vague-singletime}. 
\end{proof}

\subsection{Proof of Theorem \ref{cor:weak-convergence-single-time}:
weak convergence}
\label{sec:25}

Let us first state a lemma computing the Gaussian integral
over currents of the graph. It is proved in appendix 
\ref{app:B}. 

\begin{lemma}
\label{le:integral-kappa-single}
The following formula holds:
\begin{align}
\label{eq:claim-lemma-integral-kappa-no-square-single}
& \int_{\cH}
\exp\Big(-\sum_{(i,j)\in\vec E} \frac{\kappa_{ij}^2}{2\omega_{ij}}
\Big) \, d\kappa_{\cH} 
= 2^{|E|-|V|+1}\pi^{|E|-\frac{|V|-1}{2}}\frac{\prod_{\{i,j\}\in E}\omega_{ij}}{
\sqrt{\sum_{S\in\T}\prod_{\{i,j\}\in S}\omega_{ij}}}.
\end{align}
\end{lemma}

We first prove that $\mu^\single_{i_0}$ is a probability measure. 
For $v\in\Omega_{i_0}$ and $T\in\T$, one has 
\begin{align}
&\sum_{i_1\in V}\int_\cH\rho_{i_0}^\single(\kappa,v,i_1,T)\, d\kappa_\cH 
= 2^{|E|-|V|+1}\pi^{-\frac{|V|-1}{2}}
\exp\Big(\sum_{\{i,j\}\in E} W_{ij}\left(
1-\cosh(v_i-v_j)\right)\Big) \nonumber\\
& \cdot\prod_{\{i,j\}\in E\setminus T}\frac{1}{2\omega_{ij}}
\cdot\frac{\prod_{\{i,j\}\in E}\omega_{ij}}{
\sqrt{\sum_{S\in\T}\prod_{\{i,j\}\in S}\omega_{ij}}}\prod_{i\in V\setminus\{i_0\}}e^{-v_i}. 
\end{align}
Summing over the spanning trees $T$ yields the expression
\begin{align}
\pi^{-\frac{|V|-1}{2}}
\exp\Big(\sum_{\{i,j\}\in E} W_{ij}\left(
1-\cosh(v_i-v_j)\right)\Big) 
\sqrt{\sum_{S\in\T}\prod_{\{i,j\}\in S}\omega_{ij}}\prod_{i\in V\setminus\{i_0\}}e^{-v_i}. 
\end{align}
Integrating this over $v$ and using $\omega_{ij}=\frac12 W_{ij}e^{v_i+v_j}$ gives 1 
by Fact \ref{fact:normalization-susy} in the appendix. 
Because vague convergence of sub-probability measures to a probability measure
implies weak convergence, Lemma~\ref{le:vague-conv-single}
yields the claimed weak convergence.

\section{Proof for double timescale}
\label{se:double-timescale}
\subsection{Proof of Theorem \ref{thm:proba-of-a-path}}
\label{se:joint-density-local-times}

In this section, the time horizons $\sigma$ and $\sigma'$ are kept fixed. 
The proof follows the same lines as the proof of Theorem 
\ref{thm:proba-of-a-path-single}. For $i\in V$, let 
$k_i'=\sum_{j\in V:\{i,j\}\in E}k_{ij}'$.
Similarly to the derivation of \eqref{eq:proba-under-consideration-single}, 
we obtain 
\begin{align}
& P_{i_0}(K_{k,\sigma,k',\sigma'}\cap L_{\sigma,\sigma'}(A)\cap 
E_{i_1,\vec T,\sigma,i_1',\vec T',\sigma'})\cr
= & |\Pi_{i_0,i_1}(k,\vec T)|\cdot|\Pi_{i_1,i_1'}(k',\vec T')|
\int_A\exp\Big(
\sum_{\{i,j\}\in E} W_{ij}\Big(1-\sqrt{1+l_i+l_i'}\sqrt{1+l_j+l_j'}\Big)\Big)\cr
& \cdot 
\prod_{i\in V\setminus\{i_1'\}}\frac{1}{\sqrt{1+l_i+l_i'}} 
\prod_{(i,j)\in\vec{E}}\left(\frac{W_{ij}}{2}\right)^{k_{ij}+k_{ij}'}
\cdot\V(k,l,i_1)\V(k',l',i_1') \prod_{i\in V\setminus\{i_0\}} dl_idl_i'
\label{eq:proba-under-consideration-double}
\end{align}
with the volume factor $\V(k,l,i_1)$ given in \eqref{eq:volume-factor-k-single}
and, using the notation from \eqref{eq:Vk}, 
\begin{align}
\label{eq:volume-factor-k-prime-double}
\V(k',l',i_1')=\cV(k_{i_1'}',l_{i_1'}')\prod_{i\in V\setminus\{i_1'\}}\cV(k_i'-1,l_i')=
\frac{(l_{i_1'}')^{k_{i_1'}'}}{(k_{i_1'}')!}
\prod_{i\in V\setminus\{i_1'\}} \frac{(l_i')^{k_i'-1}}{(k_i'-1)!},
\end{align}
which is obtained by integration over the jump times between times 
$\sigma$ and $\sigma+\sigma'$. In analogy to \eqref{eq:factor-kl-single}, 
we obtain 
\begin{align}
\label{eq:factor-kl-prime-double}
|\Pi_{i_1,i_1'}(k',\vec T')| \V(k',l',i_1')
=\prod_{(i,j)\in\vec{E}}\frac{(l_i')^{k_{ij}'}}{(k_{ij}')!}
\prod_{(i,j)\in \vec T'}\frac{k_{ij}'}{l_i'}.
\end{align}
Inserting the last identity and \eqref{eq:factor-kl-single} in 
\eqref{eq:proba-under-consideration-double}, the claim follows.

\subsection{Double timescale asymptotics of the density of a path}
\label{se:asymptotics}

\paragraph{Scales of the variables.} 
Recall the definition \eqref{eq:def-omega} of $\omega_{ij}$ and 
$\omega_{ij}'$ and 
formula \eqref{eq:kij-in-terms-of-omega-single}. Its primed variant
is given by,
cf.\ \eqref{eq:def-kappa-prime} and \eqref{eq:def-u-v},
\begin{align}
\label{eq:kij-in-terms-of-omega}
k_{ij}'=l_{i_0}'\omega_{ij}'+\sqrt{l_{i_0}'}\kappa_{ij}', \quad
\frac{\sigma'}{l'_{i_0}}=\sum_{i\in V} e^{2u_i}.
\end{align}
By \eqref{eq:kij-in-terms-of-omega-single}, $\sigma$ and $l_{i_0}$ live on 
the same scale when all $v_i$ are  
bounded. A similar statement holds for $\sigma'$, $l_{i_0}'$, and $u_i$. 
Consequently, on the event $B_{\sigma,\sigma'}(M)$ defined in \eqref{eq:def-event-B}, 
all $l_i$ and all $k_{ij}$ have the same order of magnitude as $\sigma$
and all $l_i'$ and all $k_{ij}'$ have the same order of magnitude as $\sigma'$. 
By the definition \eqref{eq:def-s} of $s_i$, one has $v_i=u_i-l_{i_0}^{-1/2}s_i$. Hence,
for any given $M>0$, on the event $B_{\sigma,\sigma'}(M)$, in the limit as 
$\sigma\to\infty$, one has 
\begin{align}
\label{eq:e-v-e-u}
e^{v_i} = e^{u_i}(1+O_M(\sigma^{-1/2})), \quad
\omega_{ij}=\omega_{ij}'(1+O_M(\sigma^{-1/2})). 
\end{align}

Next, we give a version of Lemma \ref{le:second-asymptotics-single}
for two timescales.

\begin{lemma}
\label{le:second-asymptotics}
For $M>0$, on the events $B_{\sigma,\sigma'}(M)$, one has the following 
in the limit as $\min\{\sigma,\sigma'\sigma^{-2}\}\to\infty$:
\begin{align}
& \exp\Big(
\sum_{\{i,j\}\in E} W_{ij}\left(1-\sqrt{1+l_i+l_i'}\sqrt{1+l_j+l_j'}\right)\Big)
\prod_{i\in V\setminus\{i_1'\}}\frac{1}{\sqrt{1+l_i+l_i'}} \cr 
= & (l_{i_0}')^{-\frac{|V|-1}{2}} \exp\Big(\sum_{\{i,j\}\in E} W_{ij}\Big(
1-\cosh(u_i-u_j)-\frac12e^{u_i+u_j}(s_i-s_j)^2\Big)\Big) \cr
& \cdot\prod_{(i,j)\in\vec{E}}\exp(-l_{i_0}\omega_{ij}-l_{i_0}'\omega_{ij}')
\prod_{i\in V\setminus\{i_1'\}}e^{-u_i} 
\cdot \left(1+O_{M,W,G}\left(\sigma^{-1/2}+\frac{\sigma^2}{\sigma'}\right)\right).
\label{eq:2nd-asymptotics}
\end{align}
\end{lemma}
\begin{proof}
During the proof, we work on the events $B_{\sigma,\sigma'}(M)$ in the limit
$\min\{\sigma,\sigma'\sigma^{-2}\}\to\infty$. Note that this implies 
$\sigma'\gg\sigma\to\infty$. Furthermore, all $\sigma'/l_i'$ and $\sigma/l_i$ 
are bounded from above and below by $M$-dependent positive constants. 
In analogy to \eqref{eq:prod-sqrt-l} we obtain 
\begin{align}
& \sqrt{1+l_i+l_i'}
= \sqrt{l_i'}+\frac{1+l_i}{2\sqrt{l_i'}}+O\left(\frac{(1+l_i)^2}{(l_i')^{3/2}}\right)
=\sqrt{l_i'}+\frac{1+l_i}{2\sqrt{l_i'}}+O_M\left(\sigma^2(\sigma')^{-3/2}\right), 
\nonumber\\
& \sqrt{1+l_i+l_i'}\sqrt{1+l_j+l_j'}
= \sqrt{l_i'l_j'}+\frac12\sqrt{\frac{l_i'}{l_j'}}(1+l_j)
+\frac12\sqrt{\frac{l_j'}{l_i'}}(1+l_i)+O_M\left(\frac{\sigma^2}{\sigma'}\right).
\end{align}
Inserting the representation \eqref{eq:def-u-v} of $l_i',l_j'$ yields 
\begin{align}
& \frac12\left(\sqrt{\frac{l_i'}{l_j'}}(1+l_j)
+\sqrt{\frac{l_j'}{l_i'}}(1+l_i)\right)
= \frac12(e^{u_i-u_j}+e^{u_j-u_i})+
\frac{l_{i_0}}{2}\left(e^{u_i-u_j+2v_j}+e^{u_j-u_i+2v_i}\right) \nonumber\\
&= \cosh(u_i-u_j)+l_{i_0} e^{v_i+v_j}\cosh((u_i-v_i)-(u_j-v_j)).
\label{eq:intermediate-step-cosh}
\end{align}
Next, we insert the definition \eqref{eq:def-s} of $s_i$ and 
replace $\cosh$ by its Taylor expansion $\cosh x=1+\frac12x^2+O(x^3)$, 
$x\to 0$. We obtain the following expression for the last summand in 
\eqref{eq:intermediate-step-cosh}:
\begin{align}
& l_{i_0} e^{v_i+v_j}\cosh\Big(\frac{s_i-s_j}{\sqrt{l_{i_0}}}\Big)
= l_{i_0} e^{v_i+v_j}\Big(1+\frac{(s_i-s_j)^2}{2l_{i_0}}
+O\Big(\frac{(s_i-s_j)^3}{l_{i_0}^{3/2}}\Big)\Big)\nonumber\\
=&\sqrt{l_il_j} + \frac12e^{v_i+v_j}(s_i-s_j)^2 + O_M\left(\sigma^{-1/2}\right).
\label{eq:first-appearence-s}
\end{align}
Using \eqref{eq:prod-omega} and its primed version, we obtain 
\begin{align}
&\exp\Big(
\sum_{\{i,j\}\in E} W_{ij}\left(1-\sqrt{1+l_i+l_i'}\sqrt{1+l_j+l_j'}\right)\Big)
\prod_{(i,j)\in\vec{E}}\exp(l_{i_0}\omega_{ij}+l_{i_0}'\omega_{ij}') 
\label{eq:expression-exp}\\
= & \exp\Big(\sum_{\{i,j\}\in E} W_{ij}\Big(
1-\cosh(u_i-u_j)-\frac12e^{v_i+v_j}(s_i-s_j)^2\Big)\Big) 
\cdot\Big(1+O_{M,W,G}\Big(\sigma^{-1/2}+\frac{\sigma^2}{\sigma'}\Big)\Big).
\nonumber
\end{align}
By \eqref{eq:e-v-e-u}, it follows 
\begin{align}
e^{v_i+v_j}(s_i-s_j)^2=e^{u_i+u_j}(s_i-s_j)^2+O_M(\sigma^{-1/2}). 
\end{align}
Furthermore, using $\sqrt{1+l_i+l_i'}=\sqrt{l_i'}(1+O_M(\sigma/\sigma'))$, we calculate
\begin{align}
\prod_{i\in V\setminus\{i_1'\}}\frac{1}{\sqrt{1+l_i+l_i'}} 
= &  \left(1+O_{M,G}\left(\frac{\sigma}{\sigma'}\right)\right)
(l_{i_0}')^{-\frac{|V|-1}{2}}
\prod_{i\in V\setminus\{i_1'\}}e^{-u_i}.
\end{align}
Combining these facts with \eqref{eq:expression-exp} completes the proof of the lemma. 
\end{proof}

The following theorem connects the density $\rho_{i_0}^\bigsusy$ in the 
Definition \ref{def:h2-extended}
of the extended version of the $\htwo$ model to the asymptotics 
of VRJP.

\begin{theorem}{\bf (Limiting joint density)}
\label{thm:claim-thm-asymptotics-collected}
Consider the setup of Definition \ref{def:events}. 
For $M>0$, on the events $B_{\sigma,\sigma'}(M)$, one has the following 
in the limit as $\min\{\sigma,\sigma'\sigma^{-2}\}\to\infty$:
\begin{align}
& P_{i_0}(K_{k,\sigma,k',\sigma'}\cap L_{\sigma,\sigma'}(A)\cap 
E_{i_1,\vec T,\sigma,i_1',\vec T',\sigma'})
\nonumber\\
= &\left(1+O_{M,W,G}\left(\sigma^{-1/2}+\frac{\sigma^2}{\sigma'}\right)\right)
\int_A\rho^\bigsusy_{i_0}\Lambda_{\sigma,\sigma',i_0}(dl\, dl')
\label{eq:claim-thm-asymptotics-collected}
\end{align}
with the following measure on $\cL_{\sigma,\sigma'}$
\begin{align}
\label{eq:def-measure-Lambda}
\Lambda_{\sigma,\sigma',i_0}(dl\, dl')
=\frac{4^{1-|V|}}{l_{i_0}^{|E|}(l_{i_0}')^{|E|+\frac{|V|-1}{2}}}
\frac{\sigma\sigma'}{l_{i_0}l_{i_0}'}
\prod_{i\in V\setminus\{i_0\}}\frac{l_{i_0}l_{i_0}'}{l_il_i'}\, dl_idl_i'
\end{align}
and the function 
$\rho^\bigsusy_{i_0}=\rho^\bigsusy_{i_0}(\kappa,\kappa',s,v,u,i_1,i_1',T,T')$ 
defined in \eqref{eq:def-rho-susy}.
\end{theorem}
\begin{proof}
Since we want to apply Theorem \ref{thm:proba-of-a-path}, we first
derive the asymptotics of the factors $\pp$ on the event 
$B_{\sigma,\sigma'}(M)$ as $\sigma'\gg\sigma\to\infty$. 
We apply Lemma \ref{le:first-asymptotics-single} and replace some 
$\omega_{ij}$'s by $\omega_{ij}'$'s. The second identity in \eqref{eq:e-v-e-u} 
allows us to do these replacements. This yields 
\begin{align}
\pp(k,l,\vec T)
= & \frac{1}{(2\pi l_{i_0})^{|E|}} 
\exp\Big(l_{i_0}\sum_{(i,j)\in\vec E}\omega_{ij}-v_{i_1}
-\sum_{(i,j)\in\vec E}\frac{\kappa_{ij}^2}{2\omega_{ij}'}\Big) \nonumber\\
& \cdot\prod_{i\in V\setminus\{i_1\}}e^{-2v_i}
\prod_{(i,j)\in\vec{E}\setminus\vec T}\frac{1}{\omega_{ij}'}
\prod_{\{i,j\}\in E}\omega_{ij}'\cdot(1+O_{M,W,G}(\sigma^{-1/2})).
\label{eq:lemma-prod-Wl-hoch-kold}
\end{align}
Note that the first occurrence of $\omega_{ij}$ in the preceding equation, 
i.e.\ in $l_{i_0}\sum_{(i,j)\in\vec E}\omega_{ij}$, is kept without replacement 
$\omega_{ij}\not\leadsto\omega_{ij}'$, as it is scaled with $l_{i_0}$.

Similarly, on the same event, one has the following as $\sigma'\to\infty$:
\begin{align}
\pp(k',l',\vec T') 
= & \frac{1}{(2\pi l_{i_0}')^{|E|}} 
\exp\Big(l_{i_0}'\!\sum_{(i,j)\in\vec E}\omega_{ij}'+u_{i_1}-u_{i_1'}
-\!\sum_{(i,j)\in\vec E}\frac{(\kappa_{ij}')^2}{2\omega_{ij}'}\Big) \nonumber\\
& \cdot\prod_{i\in V\setminus\{i_1'\}}e^{-2u_i}
\prod_{(i,j)\in\vec{E}\setminus\vec T'}\frac{1}{\omega_{ij}'}
\prod_{\{i,j\}\in E}\omega_{ij}'\cdot(1+O_{M,W,G}((\sigma')^{-1/2})).
\label{eq:lemma-prod-Wl-prime-hoch-kold}
\end{align}
The last formula is proved using the same arguments as in the proof 
of Lemma \ref{le:first-asymptotics-single}, 
cf.\ Appendix~\ref{appendix:combinatorial-factors}, 
with the event $B_\sigma(M)$ replaced by $B_{\sigma,\sigma'}(M)$. 
Furthermore, 
$\sigma$, $k_{ij}$, $l_i$, $\kappa_{ij}$, $v_i$, $\nabla v$, $\omega_{ij}$, 
$\vec T$ used in that proof are replaced by 
$\sigma'$, $k_{ij}'$, $l_i'$, $\kappa_{ij}'$, $u_i$, $\nabla u=u_{i_1}-u_{i_1'}$, 
$\omega_{ij}'$, $\vec T'$, respectively. 
In particular, $l_{i_0}$ is replaced by $l_{i_0}'$, and 
the limit $\sigma\to\infty$ is replaced by $\sigma'\to\infty$.
We remark that $k'\in\K^+_{i_1,i_1'}$ satisfies the Kirchhoff rules 
\begin{align}
\label{eq:Kirchhoff-prime}
\sum_{\substack{j\in V:\\ \{i,j\}\in E}}(k_{ij}'-k_{ji}')=\delta_{i_1}(i)-\delta_{i_1'}(i), \quad
i\in V.
\end{align}
Hence, the equation analogous to \eqref{eq:with-kirchhoff-simplified}
for the proof of \eqref{eq:lemma-prod-Wl-prime-hoch-kold} reads as follows:
\begin{align}
& \sqrt{l_{i_0}'}\sum_{(i,j)\in\vec{E}}\kappa_{ij}'(u_i-u_j)
= \sum_{i\in V}u_i (\delta_{i_1}(i)-\delta_{i_1'}(i)) =\nabla u.
\end{align}

Substituting 
formula \eqref{eq:2nd-asymptotics} from Lemma \ref{le:second-asymptotics} 
and the formulas \eqref{eq:lemma-prod-Wl-hoch-kold}
and \eqref{eq:lemma-prod-Wl-prime-hoch-kold} for $\pp$ 
into the assertion \eqref{eq:density-path} of Theorem \ref{thm:proba-of-a-path}
yields 
\begin{align}
&P_{i_0}(K_{k,\sigma,k',\sigma'}\cap L_{\sigma,\sigma'}(A)\cap E_{i_1,\vec T,\sigma,i_1',\vec T',\sigma'})
\nonumber\\
=& \int_A (l_{i_0}')^{-\frac{|V|-1}{2}} \exp\Big(\sum_{\{i,j\}\in E} W_{ij}\Big(
1-\cosh(u_i-u_j)-\frac12e^{u_i+u_j}(s_i-s_j)^2\Big)\Big) 
\cdot\prod_{i\in V\setminus\{i_1'\}}e^{-u_i} \nonumber\\
& \cdot \Big(1+O_{M,W,G}\Big(\sigma^{-1/2}+\frac{\sigma^2}{\sigma'}\Big)\Big)
\cdot\frac{1}{(2\pi l_{i_0})^{|E|}} 
\exp\Big(-v_{i_1}
-\sum_{(i,j)\in\vec E}\frac{\kappa_{ij}^2}{2\omega_{ij}'}\Big) 
\cdot\prod_{i\in V\setminus\{i_1\}}e^{-2v_i}\nonumber\\
&\cdot\prod_{(i,j)\in\vec{E}\setminus\vec T}\frac{1}{\omega_{ij}'}
\prod_{\{i,j\}\in E}\omega_{ij}'\cdot(1+O_{M,W,G}(\sigma^{-1/2})) 
\cdot\frac{1}{(2\pi l_{i_0}')^{|E|}} 
\exp\Big(u_{i_1}-u_{i_1'}
-\sum_{(i,j)\in\vec E}\frac{(\kappa_{ij}')^2}{2\omega_{ij}'}\Big) \nonumber\\
& \cdot\prod_{i\in V\setminus\{i_1'\}}e^{-2u_i}
\prod_{(i,j)\in\vec{E}\setminus\vec T'}\frac{1}{\omega_{ij}'}
\prod_{\{i,j\}\in E}\omega_{ij}'\cdot(1+O_{M,W,G}((\sigma')^{-1/2}))
\prod_{i\in V\setminus\{i_0\}} \, dl_idl'_i.
\label{eq:huge-monster}
\end{align}
The error term $1+O_{M,W,G}\left(\sigma^{-1/2}+\frac{\sigma^2}{\sigma'}\right)$
dominates all other error terms in this formula. In analogy to 
\eqref{eq:prod-trees-single}, one has 
\begin{align}
& \prod_{(i,j)\in\vec{E}\setminus\vec T}\frac{1}{\omega_{ij}'}
\prod_{\{i,j\}\in E}\omega_{ij}'
= \prod_{\{i,j\}\in T}\omega_{ij}'\prod_{\{i,j\}\in E}\frac{1}{\omega_{ij}'}.
\label{eq:prod-trees}
\end{align}
Collecting the factors $e^{-u_i}$ and $e^{-v_i}$ and using $u_{i_0}=v_{i_0}=0$ 
and \eqref{eq:e-v-e-u} gives 
\begin{align}
& \prod_{i\in V\setminus\{i_1'\}}e^{-u_i} \cdot\exp(-v_{i_1})
\prod_{i\in V\setminus\{i_1\}}e^{-2v_i}\cdot\exp(u_{i_1}-u_{i_1'})
\prod_{i\in V\setminus\{i_1'\}}e^{-2u_i}\cr
= & e^{u_{i_1}+v_{i_1}+2u_{i_1'}} \prod_{i\in V\setminus\{i_0\}}e^{-3u_i-2v_i} 
= e^{2v_{i_1}+2u_{i_1'}} \prod_{i\in V\setminus\{i_0\}}e^{-3u_i-2v_i} 
(1+O_M(\sigma^{-1/2})) .
\label{eq:balance-exp-u}
\end{align}
Substituting \eqref{eq:prod-trees} and \eqref{eq:balance-exp-u} 
into \eqref{eq:huge-monster} and simplifying the remaining terms yields 
\begin{align}
&P_{i_0}(K_{k,\sigma,k',\sigma'}\cap L_{\sigma,\sigma'}(A)\cap E_{i_1,T,\sigma,i_1',T',\sigma'})
= \Big(1+O_{M,W,G}\Big(\sigma^{-1/2}+\frac{\sigma^2}{\sigma'}\Big)\Big)
\nonumber\\
& \cdot \int_A\frac{l_{i_0}^{-|E|}(l_{i_0}')^{-\frac{|V|-1}{2}-|E|}}{(2\pi)^{2|E|}} 
\exp\Big(\sum_{\{i,j\}\in E} W_{ij}\Big(
1-\cosh(u_i-u_j)-\frac12e^{u_i+u_j}(s_i-s_j)^2\Big)\Big) \nonumber\\
& \cdot \prod_{\{i,j\}\in T}\omega_{ij}'
\prod_{\{i,j\}\in T'}\omega_{ij}'\prod_{\{i,j\}\in E}\frac{1}{(\omega_{ij}')^2}
\cdot\exp\Big(
-\sum_{(i,j)\in\vec E} \frac{\kappa_{ij}^2+(\kappa_{ij}')^2}{2\omega_{ij}'}
\Big) \nonumber\\
& \cdot 
e^{2v_{i_1}+2u_{i_1'}} \prod_{i\in V\setminus\{i_0\}}e^{-3u_i-2v_i} \, dl_idl_i'.
\end{align}
Using Definition \ref{def:h2-extended} of $\rho_{i_0}^\bigsusy$ and the relations 
\eqref{eq:def-v}, \eqref{eq:def-u-v}, \eqref{eq:kij-in-terms-of-omega-single}, 
and \eqref{eq:kij-in-terms-of-omega},
claim \eqref{eq:claim-thm-asymptotics-collected} follows. 
\end{proof}

\subsection{Continuum limit and vague convergence}
\label{se:continuum-limit}

The main result in this section, stated in Corollary 
\ref{cor:vague-convergence-kappa-etc} below, deals with a vague convergence 
of the random vector in \eqref{eq:vector-kappa-etc}. This requires two 
ingredients. First, we need to calculate the Jacobian of the transformation 
$(s,u)\mapsto(l,l')$. Second, we deal with convergence of a Riemann sum indexed 
by $(\kappa(\sigma),\kappa'(\sigma,\sigma'))$ to an integral. 
These two ingredients are treated in the following two lemmas. 

Recall the variables $s,u,v$ from \eqref{eq:def-s}, \eqref{eq:def-u-v},
and \eqref{eq:def-v}, written in the form 
\begin{align}
\label{eq:s-u-in-terms-of-l-l-prime}
s_i=\frac{\sqrt{l_{i_0}}}{2}\left(\log\frac{l_i'}{l_{i_0}'}-\log\frac{l_i}{l_{i_0}}\right), \quad
u_i=\frac12\log\frac{l_i'}{l_{i_0}'}, \quad
v_i=\frac12\log\frac{l_i}{l_{i_0}}.
\end{align}
The first lemma considers the transformation 
$g_{\sigma,\sigma'}:\cL_{\sigma,\sigma'}\to\Omega_{i_0}^2$, $(l,l')\mapsto (s,u)$.
We restrict it to the following variant 
$\tilde B_{\sigma,\sigma'}(M)$ of the event $B_{\sigma,\sigma'}(M)$, cf.\
\eqref{eq:def-event-B}:
\begin{align}
\tilde B_{\sigma,\sigma'}(M)
=& \left\{ (l,l')\in\cL_{\sigma,\sigma'}:\, 
|s_i(l,l')|,|u_i(l,l')|,|v_i(l,l')|\le M\text{ for all }i\in V\right\}.
\end{align}

\begin{lemma}{\bf (Jacobian of the variable transformation)}
\label{le:jacobian}
For $M,\sigma'>0$ and $\sigma>|V|M^2e^{2M}$, the map $g_{\sigma,\sigma'}$
is a bijection between $\tilde B_{\sigma,\sigma'}(M)$ and its range. 
The following formula describes the corresponding transformation of 
measure:
\begin{align}
\label{eq:change-prod-dli}
(l_{i_0}l_{i_0}')^{|E|+\frac{1-|V|}{2}}
\Lambda_{\sigma,\sigma',i_0}(dl\, dl') 
=& (4\sqrt{l_{i_0}}l_{i_0}')^{1-|V|}\frac{\sigma\sigma'}{l_{i_0}l_{i_0}'}
\prod_{i\in V\setminus\{i_0\}} \frac{l_{i_0}l_{i_0}'}{l_il_i'} \, dl_idl_i'
\nonumber\\
= &\frac{1}{h_\sigma(s,u)}\prod_{i\in V\setminus\{i_0\}} ds_idu_i ,\\
\text{where}\qquad 
h_\sigma(s,u)=& 1+\frac{1}{\sigma\sqrt{l_{i_0}}}\sum_{i\in V\setminus\{i_0\}}l_is_i.
\end{align}
Here, $l=l(s,u)$ denotes the first component of $g_{\sigma,\sigma'}^{-1}(s,u)$.
For all $(s,u)\in g_{\sigma,\sigma'}[\tilde B_{\sigma,\sigma'}(M)]$ 
the expression $h_\sigma(s,u)$ fulfills the bound 
\begin{align}
\label{eq:estimate-h}
|h_\sigma(s,u)-1|\le \frac{\sqrt{|V|}Me^M}{\sqrt\sigma}<1.
\end{align} 
\end{lemma}
\begin{proof}
Given $(s,u)\in[-M,M]^{V\times V}\cap\operatorname{range} g_{\sigma,\sigma'}$,  
we have to show that it has a \emph{unique} inverse image 
$(l,l')\in\tilde B_{\sigma,\sigma'}(M)$. First, $l_{i_0}'$ is uniquely determined
by $l_{i_0}'=\sigma'/\sum_{i\in V}e^{2u_i}$.
Second, $l_i'=l_{i_0}'e^{2u_i}$, $i\in V$, shows that $l'$ is also uniquely
determined. Third, 
\begin{align}
\sigma=\sum_{i\in V}l_i
=l_{i_0}\sum_{i\in V}\exp\Big(2u_i-\frac{2s_i}{\sqrt{l_{i_0}}}\Big)
\end{align}
gives us a transcendental equation for $l_{i_0}$. It determines $l_{i_0}$
uniquely because 
\begin{align}
\frac{\partial}{\partial l_{i_0}}\Big[
l_{i_0}\sum_{i\in V}\exp\Big(2u_i-\frac{2s_i}{\sqrt{l_{i_0}}}\Big)
\Big]
=\sum_{i\in V}\Big(1+\frac{s_i}{\sqrt{l_{i_0}}}\Big)
\exp\Big(2u_i-\frac{2s_i}{\sqrt{l_{i_0}}}\Big)>0;
\end{align}
here we use that by our choice of $\sigma$
\begin{align}
\label{eq:est-s-over-l}
l_{i_0}=\frac{\sigma}{\sum_{i\in V}e^{2v_i}}\ge \frac{\sigma}{|V|e^{2M}}
\quad\Rightarrow\quad 
\frac{|s_i|}{\sqrt{l_{i_0}}}\le\frac{\sqrt{|V|}Me^M}{\sqrt\sigma}<1.
\end{align}
This implies that $v_i=u_i-s_i/\sqrt{l_{i_0}}$, $i\in V$, is uniquely 
determined, too. Finally, all $l_i$, $i\in V$, are also uniquely 
determined because of $l_i=l_{i_0}e^{2v_i}$. 
Summarizing, we have shown that $g_{\sigma,\sigma'}$ restricted to
$\tilde B_{\sigma,\sigma'}(M)$ is one-to-one. 
To calculate the Jacobi determinant of the map 
$(l_i,l_i')_{i\in V\setminus\{i_0\}}\mapsto(s_i,u_i)_{i\in V\setminus\{i_0\}}$ 
one observes in analogy to \eqref{eq:dv-dl} for $i,j\in V\setminus\{i_0\}$, 
using \eqref{eq:s-u-in-terms-of-l-l-prime}:
\begin{align}
\frac{\partial s_i}{\partial l_j}
=-\frac{\sqrt{l_{i_0}}}{2l_i}\Big(\delta_{ij}+\frac{l_i}{l_{i_0}}
\Big(1+\frac{s_i}{\sqrt{l_{i_0}}}\Big)\Big), \quad 
\frac{\partial u_i}{\partial l_j'}
=\frac{1}{2l_i'}\Big(\delta_{ij}+\frac{l_i'}{l_{i_0}'}\Big), \quad
\frac{\partial u_i}{\partial l_j}=0.
\end{align}
In particular, $\partial s/\partial l$ and $\partial u/\partial l'$ 
are rank 1 perturbations of invertible diagonal matrices. Using that 
$\det(\Id+A)=1+\operatorname{tr}(A)$ for rank 1 matrices $A$, we obtain,
cf.\ formula \eqref{eq:det-dv-dl}:
\begin{align}
\left|\det \frac{\partial s}{\partial l}\right|
= &\Big(1+
\sum_{i\in V\setminus\{i_0\}}\frac{l_i}{l_{i_0}}
\Big(1+\frac{s_i}{\sqrt{l_{i_0}}}\Big)
\Big)
\prod_{i\in V\setminus\{i_0\}}\frac{\sqrt{l_{i_0}}}{2l_i}
=\frac{\sigma}{l_{i_0}}h_\sigma(s,u)
\prod_{i\in V\setminus\{i_0\}}\frac{\sqrt{l_{i_0}}}{2l_i}, 
\nonumber\\
\left|\det \frac{\partial u}{\partial l'}\right|
=&\Big(1+
\sum_{i\in V\setminus\{i_0\}}\frac{l_i'}{l_{i_0}'}\Big)
\prod_{i\in V\setminus\{i_0\}}\frac{1}{2l_i'}
=\frac{\sigma'}{l_{i_0}'}\prod_{i\in V\setminus\{i_0\}}\frac{1}{2l_i'}. 
\end{align}
In view of $\partial u/\partial l=0$, it follows 
\begin{align}
\left|\det \frac{\partial (s,u)}{\partial(l,l')}\right|
=\left|\det \frac{\partial s}{\partial l}\right|\cdot 
\left|\det \frac{\partial u}{\partial l'}\right|
= |h_\sigma(s,u)|(4\sqrt{l_{i_0}}l_{i_0}')^{1-|V|}\frac{\sigma\sigma'}{l_{i_0}l_{i_0}'}
\prod_{i\in V\setminus\{i_0\}} \frac{l_{i_0}l_{i_0}'}{l_il_i'}. 
\end{align}
As soon as we know $h_\sigma(s,u)>0$, it follows that the Jacobi determinant
does not vanish and that claim \eqref{eq:change-prod-dli} holds. 

Positivity of $h_\sigma(s,u)$, used in the previous arguments, is an immediate
consequence of \eqref{eq:estimate-h}, which is proven as follows. 
Using \eqref{eq:est-s-over-l}, which works under the assumption 
$\sigma>|V|M^2e^{2M}$, we estimate
\begin{align}
|h_\sigma(s,u)-1|
\le \frac{1}{\sigma}\sum_{i\in V\setminus\{i_0\}}l_i\frac{|s_i|}{\sqrt{l_{i_0}}}
< \frac{1}{\sigma}\sum_{i\in V\setminus\{i_0\}}l_i
<1.
\end{align}
\end{proof}

\medskip
Let $i_0,i_1,i_1'\in V$ and $\sigma,\sigma'>0$. 
Recall the definitions of $\Omega_{i_0}$, $\K_{i_0,i_1,i_1'}$, $\K_{i_0,i_1,i_1'}^+$, 
and $\cL_\sigma$ stated in \eqref{eq:def-Omega}, \eqref{eq:def-K}, and 
\eqref{eq:def-L}, respectively, and the notation 
$\cL_{\sigma,\sigma'}=\cL_\sigma\times\cL_{\sigma'}$. We consider the measure 
\begin{align}
\label{eq:def-lambda-plus}
\lambda^+_{\sigma,\sigma',i_0,i_1,i_1'}=
\sum_{(k,k')\in\K_{i_0,i_1,i_1'}^+}\delta_k\delta_{k'}\,
\Lambda_{\sigma,\sigma',i_0}(dl\, dl') 
\end{align}
defined on $\K_{i_0,i_1,i_1'}\times\cL_{\sigma,\sigma'}$. 
Let $\lambda_{\sigma,\sigma',i_0,i_1,i_1'}$ be defined as 
$\lambda^+_{\sigma,\sigma',i_0,i_1,i_1'}$ with the only
difference that the summation over $\K_{i_0,i_1,i_1'}^+$ is replaced by 
$\K_{i_0,i_1,i_1'}$. 
We introduce the following variant of the map $F_{\sigma,\sigma',i_0}$, cf.\ 
\eqref{eq:def-F-sigma-sigma'-i0}:
\begin{align}
F_{\sigma,\sigma',i_0,i_1,i_1'}:\K_{i_0,i_1,i_1'}\times\cL_{\sigma,\sigma'}
& \to (\R^{\vec E})^2\times\Omega_{i_0}^3 ,\cr
(k,k',l,l')& \mapsto (\kappa,\kappa',s,v,u)
\end{align}
using again the equations \eqref{eq:def-v}, \eqref{eq:def-kappa}, and 
\eqref{eq:def-u-v}--
\eqref{eq:def-kappa-prime}.

\begin{lemma}{\bf (Vague convergence of the reference measure)}
\label{le:limit-image-measure}\\
The image measure $F_{\sigma,\sigma',i_0,i_1,i_1'}[\lambda^+_{\sigma,\sigma',i_0,i_1,i_1'}]$
converges vaguely as $\sigma,\sigma'\to\infty$ to 
\begin{align}
d\kappa_\cH\,d\kappa_\cH'\, \prod_{i\in V\setminus\{i_0\}} 1_{\{u_i=v_i\}}\, ds_i\, du_i; 
\end{align}
recall that $1_{\{u_i=v_i\}}du_i$ denotes the Lebesgue measure on the diagonal of $\R^2$.
In other words, for any continuous compactly supported test function 
$f:(\R^{\vec E})^2\times\Omega_{i_0}^3\to\R$, one has 
\begin{align}
& \lim_{\sigma,\sigma'\to\infty}\int\limits_{\K_{i_0,i_1,i_1'}\times\cL_{\sigma,\sigma'}} 
f\circ F_{\sigma,\sigma',i_0,i_1,i_1'}\, d\lambda_{\sigma,\sigma',i_0,i_1,i_1'}^+ 
=  \int\limits_{\Omega_{i_0}^2}\int\limits_{\cH^2} 
f(\kappa,\kappa',s,u,u)\, d\kappa_\cH d\kappa_\cH'\,
ds_i\, du_i . 
\label{eq:claim-vague-conv-test-fns1}
\end{align}
\end{lemma}
\begin{proof}
The proof relies on the same technique as in Lemma \ref{le:vague-conv-single}.
Given a test function $f$ as in the assumption, we claim that there exists 
$\varepsilon>0$, depending on $f$, such that 
\begin{align}
\label{eq:F-inverse-suppf}
F_{\sigma,\sigma',i_0,i_1,i_1'}^{-1}[\supp f]\subseteq 
\K_{i_0,i_1,i_1'}^+\times(\sigma\varepsilon,\sigma)^V\times
(\sigma'\varepsilon,\sigma')^V 
\end{align}
holds for $\sigma,\sigma'$ large enough. 
Indeed, there is a constant $M>0$, depending on $\supp f$, such that 
for any $(k,k',l,l')\in F_{\sigma,\sigma',i_0,i_1,i_1'}^{-1}[\supp f]$, 
all components of 
$F_{\sigma,\sigma',i_0,i_1,i_1'}(k,k',l,l')=:(\kappa,\kappa',s,v,u)$
are bounded in absolute value by $M$. Similarly as in the beginning of the 
proof of Lemma \ref{le:vague-conv-single}, $k_{ij},k_{ij}'>0$ and all 
$l_i/\sigma$ and $l_i'/\sigma'$ are bounded away from $0$ for $\sigma,\sigma'$
large enough. This proves \eqref{eq:F-inverse-suppf}. 
Consequently, vague convergence of the image measure
$F_{\sigma,\sigma',i_0,i_1,i_1'}[\lambda^+_{\sigma,\sigma',i_0,i_1,i_1'}]$ 
as $\sigma,\sigma'\to\infty$ is equivalent to vague convergence of 
$F_{\sigma,\sigma',i_0,i_1,i_1'}[\lambda_{\sigma,\sigma',i_0,i_1,i_1'}]$ to the same
limit. Next, we prove the latter one.

Fix a path from $i_0$ to $i_1$ and another one from $i_1$ to $i_1'$. 
Let $\pi=(\pi_{ij})_{(i,j)\in\vec E}$ respectively $\pi'=(\pi_{ij}')_{(i,j)\in\vec E}$
be the corresponding edge crossing numbers. 
We introduce the shift vectors 
$\vartheta_l=(l_{i_0}^{1/2}\omega_{ij})_{(i,j)\in\vec E}$,
$\vartheta_{l'}=(l_{i_0}'^{1/2}\omega_{ij}')_{(i,j)\in\vec E}\in\cH$, 
cf.\ \eqref{eq:def-omega}. 
Let $\Gamma\subset\cH$ denote the lattice which has $\Z^{\vec E\setminus\vec T_0}$
as its image under the restriction map $\R^{\vec E}\to\R^{\vec E\setminus\vec T_0}$.
When $(k,k')$ runs over $\K_{i_0,i_1,i_1'}$, the corresponding 
$(\kappa-l_{i_0}^{-1/2}\pi,\kappa'-(l_{i_0}')^{-1/2}\pi')$ runs over the shifted 
lattice
$(l_{i_0}^{-1/2}\Gamma-\vartheta_l)\times ((l_{i_0}')^{-1/2}\Gamma-\vartheta_{l'})$.
In other words, for any 
$(l,l')\in\cL_{\sigma,\sigma'}$, one has 
\begin{align}
\label{eq:riemann-sum-f-circ-F}
&\sum_{(k,k')\in\K_{i_0,i_1,i_1'}} f(F_{\sigma,\sigma',i_0,i_1,i_1'}(k,k',l,l'))
=\sum_{\substack{\tilde\kappa\in l_{i_0}^{-1/2}\Gamma-\vartheta_l\\
\tilde\kappa'\in(l_{i_0}')^{-1/2}\Gamma-\vartheta_{l'}}}
f\Big(\tilde\kappa+\frac{\pi}{\sqrt{l_{i_0}}},
\tilde\kappa'+\frac{\pi'}{\sqrt{l_{i_0}'}},s,v,u\Big),
\end{align}
where $s,v,u$ are given in \eqref{eq:s-u-in-terms-of-l-l-prime}.
Integrating first \eqref{eq:riemann-sum-f-circ-F} over $l$ and $l'$ with 
appropriate weights and using 
Lemma \ref{le:jacobian} in the second equality, we obtain 
\begin{align}
\label{eq:integral-with-subst}
& \int_{\K_{i_0,i_1,i_1'}\times\cL_{\sigma,\sigma'}} f\circ F_{\sigma,\sigma',i_0,i_1,i_1'}\, 
d\lambda_{\sigma,\sigma',i_0,i_1,i_1'} \\
= & \int_{\cL_{\sigma,\sigma'}} 
\sum_{\substack{\tilde\kappa\in l_{i_0}^{-1/2}\Gamma-\vartheta_l\\
\tilde\kappa'\in(l_{i_0}')^{-1/2}\Gamma-\vartheta_{l'}}}
f\Big(\tilde\kappa+\frac{\pi}{\sqrt{l_{i_0}}},
\tilde\kappa'+\frac{\pi'}{\sqrt{l_{i_0}'}},s,v,u\Big)
\Big|_{\substack{s,v,u\\ \text{ from \eqref{eq:s-u-in-terms-of-l-l-prime}}}} 
\Lambda_{\sigma,\sigma',i_0}(dl\, dl') \nonumber\\
= & \int_{\Omega_{i_0}^2}\frac{1}{(l_{i_0}l_{i_0}')^{\frac{|\vec E\setminus\vec T_0|}{2}}}
\hspace{-1mm}
\sum_{\substack{\tilde\kappa\in l_{i_0}^{-1/2}\Gamma-\vartheta_l\\
\tilde\kappa'\in(l_{i_0}')^{-1/2}\Gamma-\vartheta_{l'}}}
f\Big(\tilde\kappa+\frac{\pi}{\sqrt{l_{i_0}}},
\tilde\kappa'+\frac{\pi'}{\sqrt{l_{i_0}'}},s,v,u\Big)
\Big|_{\substack{l_{i_0},l_{i_0}',v\\ \text{ from \eqref{eq:subst-l,l-prime,v}}}} 
 \frac{\prod_{i\in V\setminus\{i_0\}}ds_idu_i}{h_\sigma(s,u)} \nonumber
\end{align}
with the substitution 
\begin{align}
\label{eq:subst-l,l-prime,v}
v_i=u_i-\frac{s_i}{\sqrt{l_{i_0}}}, \quad
l_{i_0}=\frac{\sigma}{\sum_{i\in V}e^{2u_i-2s_i/\sqrt{l_{i_0}}}}, \quad
l_{i_0}'=\frac{\sigma'}{\sum_{i\in V}e^{2u_i}}.
\end{align}
Note that $|u_i-v_i|\le M/\sqrt{l_{i_0}}$ holds, whenever the integrand in 
\eqref{eq:integral-with-subst} is non-zero. 
We interpret the Riemann sum in \eqref{eq:integral-with-subst} as an 
integral over functions which are constant on boxes associated 
to a shifted version of 
$l_{i_0}^{-1/2}\Gamma\times(l_{i_0}')^{-1/2}\Gamma$. These boxes have volume
$(l_{i_0}l_{i_0}')^{\frac{|\vec E\setminus\vec T_0|}{2}}$. Using the dominated convergence 
theorem and the bound \eqref{eq:estimate-h} to perform the limit 
$\sigma,\sigma'\to\infty$, 
the claim \eqref{eq:claim-vague-conv-test-fns1} follows. 
\end{proof}

We endow the set $\cO_{\sigma,\sigma',i_0}$ defined in \eqref{eq:def-event-O-i0}
with the measure $\lambda_{\sigma,\sigma',i_0}^+$ 
which is characterized as follows. When we restrict $\lambda_{\sigma,\sigma',i_0}^+$ 
to $\K_{i_0,i_1,i_1'}
\times\cL_{\sigma,\sigma'}\times\{i_1\}\times\{i_1'\}\times\{T\}\times\{T'\}$
for any $i_i,i_1'\in V$, $T\in\T_{i_1}$, and $T'\in\T_{i_1'}$ and project
it down to $\K_{i_0,i_1,i_1'}\times\cL_{\sigma,\sigma'}$, it becomes 
$\lambda_{\sigma,\sigma',i_0,i_1,i_1'}^+$. Recall the definition 
\eqref{eq:def-F-sigma-sigma'-i0} of the map $F_{\sigma,\sigma',i_0}$ and  
the definition \eqref{eq:def-xi} of the random variable 
$\xi_{\sigma,\sigma'}$. Theorem \ref{thm:claim-thm-asymptotics-collected} and the last 
lemma are combined in the following corollary.

\begin{corollary}{\bf (Vague convergence to the extended $\htwo$ model)}
\label{cor:vague-convergence-kappa-etc}
The joint sub-probability distribution of
\begin{align}
& (\kappa(\sigma),\kappa'(\sigma,\sigma'),s(\sigma,\sigma'),v(\sigma), 
u(\sigma,\sigma'),
Z_\sigma,Z_{\sigma+\sigma'},T^\lastexit(0,\sigma),T^\lastexit(\sigma,\sigma+\sigma'))
\end{align}
with respect to $P_{i_0}(\cdot\cap\{\xi_{\sigma,\sigma'}\in\cO_{\sigma,\sigma',i_0}\})$ converges 
vaguely as $\min\{\sigma,\sigma'\sigma^{-2}\}\to\infty$ to $\mu^\bigsusy_{i_0}$. 
In other words, for any continuous compactly supported test function 
$f:(\R^{\vec E})^2\times\Omega_{i_0}^3\times V^2\times\T^2\to\R$, one has 
\begin{align}
\lim_{\min\{\sigma,\sigma'\sigma^{-2}\}\to\infty}
E_{i_0}\left[f(F_{\sigma,\sigma',i_0}(\xi_{\sigma,\sigma'})),\xi_{\sigma,\sigma'}\in\cO_{\sigma,\sigma',i_0}\right]
=\int\limits_{\cH^2\times\Omega_{i_0}^3\times V^2\times\T^2}
f\, d\mu^\bigsusy_{i_0}.
\end{align}
\end{corollary}
\begin{proof}
Because $f$ is compactly supported, we can choose a constant $M>0$ such 
that for any $(\kappa,\kappa',s,v,u,i_1,i_1',T,T')\in\supp f$
all the components of $\kappa,\kappa',s,v,u$ are bounded in absolute value
by $M$. 
Theorem \ref{thm:claim-thm-asymptotics-collected} yields in the limit as 
$\min\{\sigma,\sigma'\sigma^{-2}\}\to\infty$
\begin{align}
& E_{i_0}\left[f(F_{\sigma,\sigma',i_0}(\xi_{\sigma,\sigma'})),\xi_{\sigma,\sigma'}\in\cO_{\sigma,\sigma',i_0}\right]
\nonumber\\
= & \left(1+O_{M,W,G}\left(\sigma^{-1/2}+\frac{\sigma^2}{\sigma'}\right)\right) 
\int_{\cO_{\sigma,\sigma',i_0}} f\rho_{i_0}^\bigsusy \, d(F_{\sigma,\sigma',i_0}[\lambda_{\sigma,\sigma',i_0}^+]) 
\end{align}
Note that the density $\rho_{i_0}^\bigsusy$ is continuous. Hence, 
Lemma \ref{le:limit-image-measure} implies that the last integral converges 
as $\min\{\sigma,\sigma'\sigma^{-2}\}\to\infty$ to the following integral.
\begin{align}
& \int\limits_{\cH^2\times\Omega_{i_0}^2\times V^2\times\T^2}
(f\rho_{i_0}^\bigsusy)(\kappa,\kappa',s,u,u,i_1,i_1',T,T')
\, d\kappa_\cH d\kappa_\cH'\, ds_i\, du_i\, di_1\, di_1'\, dT \, dT'  
\nonumber\\
= & \int\limits_{\cH^2\times\Omega_{i_0}^3\times V^2\times\T^2}
f\, d\mu^\bigsusy_{i_0},
\end{align}
where we used the definition \eqref{eq:def-mu-susy} of $\mu^\bigsusy_{i_0}$ 
in the last step. This proves the claim. 
\end{proof}

\subsection{Marginals and weak convergence}
\label{se:marginals}

In this section, we calculate marginals of the extended $\htwo$ measure 
$\mu_{i_0}^\bigsusy$ by integrating out the current vectors $\kappa$ and $\kappa'$,
summing over the endpoints $i_1$ and $i_1'$ of paths, and summing over the
spanning tree $T$. 
The main theorems follow now easily by combining the previous results:

\medskip\noindent
\begin{proof}[Proof of Theorem \ref{thm:marginal-rho-susy}]
Combining Lemma \ref{le:integral-kappa-single} with the Definition 
\ref{def:h2-extended} of $\mu^\bigsusy_{i_0}$, 
the marginal of $(s,u,i_1,i_1',T,T')$ with respect to $\mu_{i_0}^\bigsusy$ 
is given by
\begin{align}
&   \frac{1}{\pi^{|V|-1}}
\exp\Big(\sum_{\{i,j\}\in E} W_{ij}\Big(
1-\cosh(u_i-u_j)-\frac12e^{u_i+u_j}(s_i-s_j)^2\Big)\Big) 
\prod_{\{i,j\}\in T'}\omega_{ij}'
\nonumber\\
& \cdot
\frac{\prod_{\{i,j\}\in T}\omega_{ij}'}{
\sum_{S\in\T}\prod_{\{i,j\}\in S}\omega_{ij}'}
\frac{e^{2u_{i_1}+2u_{i_1'}}}{\big(\sum_{j\in V}e^{2 u_j}\big)^2}
\prod_{i\in V\setminus\{i_0\}}e^{-u_i}\, ds_i\, du_i
\cdot di_1\, di_1'\, dT\, dT'
\label{eq:marginal1-rho-susy}
\end{align}
Note that the following holds: 
\begin{align}
\prod_{\{i,j\}\in T}\omega_{ij}'=2^{-(|V|-1)}\prod_{\{i,j\}\in T}W_{ij}e^{u_i+u_j}
\end{align}
Summing over $i_1,i_1'\in V$ and $T\in\T$, claim 
\eqref{eq:marginal2-big-susy} follows. 
Using this and  the fact that $\mu^\susy_{i_0}$ is a probability measure, cf.\ 
\eqref{eq:normalizing-const}, imply that $\mu^\bigsusy_{i_0}$ is a 
probability measure as well. 
\end{proof}

\medskip\noindent
\begin{proof}[Proof of Theorem \ref{thm:weak-convergence}]
By Theorem \ref{thm:marginal-rho-susy}, 
the measure $\mu_{i_0}^\bigsusy$ is a probability measure. 
Because vague convergence of sub-probability measures to a probability measure
implies weak convergence, Corollary \ref{cor:vague-convergence-kappa-etc} 
yields the claimed weak convergence. For the constant test function $f=1$,
the last claim follows. 
\end{proof}

\medskip\noindent
\begin{proof}[Proof of Theorem \ref{thm:second-marginal-rho-susy}]
On the one hand, by Theorem \ref{cor:weak-convergence-single-time}, 
the law of the reduced vector 
$(\kappa(\sigma),v(\sigma),Z_\sigma,T^\lastexit(0,\sigma))$
with respect to the 
sub-probability measure $P_{i_0}(\cdot\cap\{\xi_\sigma\in\cQ_{\sigma,i_0}\})$
converges weakly as $\sigma\to\infty$ to $\mu_{i_0}^\single$. 

On the other hand, by Theorem \ref{thm:weak-convergence}, the same vector 
converges weakly to the marginal $\cL_{\mu_{i_0}^\bigsusy}(\kappa,v,i_1,T)$ 
as $\min\{\sigma,\sigma'\sigma^{-2}\}\to\infty$ with respect to the 
sub-probabi\-li\-ty measure $P_{i_0}(\cdot\cap\{\xi_{\sigma,\sigma'}\in\cO_{\sigma,\sigma',i_0}\})$. 
Because of 
$\{\xi_{\sigma,\sigma'}\in\cO_{\sigma,\sigma',i_0}\}\subseteq\{\xi_\sigma\in\cQ_{\sigma,i_0}\}$
we have the same weak limit $\cL_{\mu_{i_0}^\bigsusy}(\kappa,v,i_1,T)$
with respect to the sub-probability measure 
$P_{i_0}(\cdot\cap\{\xi_\sigma\in\cQ_{\sigma,i_0}\})$, again as 
$\min\{\sigma,\sigma'\sigma^{-2}\}\to\infty$. However, the second time scale
$\sigma'$ does not play any role in the last statement anymore. Hence, we 
may replace the limit $\min\{\sigma,\sigma'\sigma^{-2}\}\to\infty$ by the 
single-time limit $\sigma\to\infty$. 

Comparing the two approaches, the weak limits agree: 
$\cL_{\mu_{i_0}^\bigsusy}(\kappa,v,i_1,T)=\mu_{i_0}^\single$. 
\end{proof}

An alternative proof of Theorem \ref{thm:second-marginal-rho-susy}
directly computes the marginal by integrating out the dropped variables. 
However, we feel that the proof presented here is simpler.

\begin{appendix}
\section{Asymptotics of the combinatorial factors}
\label{appendix:combinatorial-factors}

\begin{proof}[Proof of Lemma \ref{le:first-asymptotics-single}]
We abbreviate
\begin{align}
\label{eq:def-nabla-v}
\nabla v=v_{i_0}-v_{i_1}=-v_{i_1}.
\end{align}
In the whole proof, we work only on the events $B_\sigma(M)$.
All Landau symbols $O$ are understood in the limit as $\sigma\to\infty$. 
Note that on $B_\sigma(M)$, we have $k_{ij}\to\infty$ for any 
$(i,j)\in\vec{E}$ as $\sigma\to\infty$, with $k_{ij}/\sigma$ being bounded 
away from $0$. Hence, by Stirling's formula, 
\begin{align}
k_{ij}!=\sqrt{2\pi}e^{-k_{ij}}k_{ij}^{k_{ij}+\frac12}(1+O_M(\sigma^{-1})).
\end{align}
For $(i,j)\in\vec E$, one has 
\begin{align}
\label{eq:part-of-density-for-one-kij}
k_{ij}\left(\frac{W_{ij}l_i}{2}\right)^{k_{ij}}\frac{1}{k_{ij}!}
= \sqrt{\frac{k_{ij}}{2\pi}} \left(\frac{W_{ij}l_ie}{2k_{ij}}\right)^{k_{ij}}
(1+O_M(\sigma^{-1})).
\end{align} 
If $(i,j)\in\vec T$, this is the $k_{ij}$-dependent part in the 
definition \eqref{eq:def-pp} of $\pp(k,l,\vec T)$. Using 
\begin{align}
\label{eq:kij-asymptotics}
k_{ij}=l_{i_0}(\omega_{ij}+l_{i_0}^{-1/2}\kappa_{ij})
= l_{i_0}\omega_{ij}
\left(1+O_{M,W}\left(\sigma^{-1/2}\right)\right)
\end{align}
and the symmetry $\omega_{ij}=\omega_{ji}$, we deduce 
\begin{align}
\label{eq:prod-sqrt-k}
\prod_{(i,j)\in\vec E} \sqrt{\frac{k_{ij}}{2\pi}} 
=\left(\frac{l_{i_0}}{2\pi}\right)^{|E|}\prod_{\{i,j\}\in E}\omega_{ij}
\cdot(1+O_{M,W,G}(\sigma^{-1/2})). 
\end{align} 
One has 
\begin{align}
\frac{W_{ij}l_ie}{2k_{ij}}
=\frac{W_{ij}l_{i_0} e^{1+2v_i}}{2(l_{i_0}\omega_{ij}+l_{i_0}^{1/2}\kappa_{ij})}
=\frac{\omega_{ij}e^{1+v_i-v_j}}{\omega_{ij}+l_{i_0}^{-1/2}\kappa_{ij}}.
\end{align} 
Consequently, we obtain 
\begin{align}
\label{eq:log-kij}
\log\left[\left(\frac{W_{ij}l_ie}{2k_{ij}}\right)^{k_{ij}}\right]
= & l_{i_0}\left(\omega_{ij}+l_{i_0}^{-1/2}\kappa_{ij}\right)
\left(\log\omega_{ij}+1+v_i-v_j\right) \cr
& -l_{i_0}\left(\omega_{ij}+l_{i_0}^{-1/2}\kappa_{ij}\right)
\log\left(\omega_{ij}+l_{i_0}^{-1/2}\kappa_{ij}\right).
\end{align} 
Using the Taylor expansion 
$x\log x= x_0\log x_0+(1+\log x_0)(x-x_0)+\frac{(x-x_0)^2}{2x_0}+O((x-x_0)^3)$
as $x\to x_0$ at $x_0=\omega_{ij}$ for the second term, we deduce 
\begin{align}
\label{eq:klogW}
& \log\left[\left(\frac{W_{ij}l_ie}{2k_{ij}}\right)^{k_{ij}}\right]
= l_{i_0}\left(\omega_{ij}+l_{i_0}^{-1/2}\kappa_{ij}\right)(\log\omega_{ij}+1+v_i-v_j)\\
& -l_{i_0}\left(\omega_{ij}\log\omega_{ij}+(1+\log\omega_{ij})l_{i_0}^{-1/2}\kappa_{ij}
+\frac{\kappa_{ij}^2}{2\omega_{ij}l_{i_0}}
+O_M\left(\sigma^{-3/2}\right)\right)\cr
= & l_{i_0}\omega_{ij}(1+v_i-v_j)+\sqrt{l_{i_0}}\kappa_{ij}(v_i-v_j)
-\frac{\kappa_{ij}^2}{2\omega_{ij}}
+O_M\left(\sigma^{-1/2}\right).\nonumber
\end{align} 
Since $\omega_{ij}=\omega_{ji}$ and with $(i,j)\in\vec{E}$ there is 
$(j,i)\in\vec{E}$ as well, we have 
\begin{align}
\label{eq:sum-omega-v-zero}
\sum_{(i,j)\in\vec{E}}\omega_{ij}(v_i-v_j)=0.
\end{align}
Note that 
\begin{align}
\sum_{(i,j)\in\vec{E}}\kappa_{ij}v_j
=\sum_{(j,i)\in\vec{E}}\kappa_{ji}v_i
=\sum_{(i,j)\in\vec{E}}\kappa_{ji}v_i.
\end{align}
Using this, Kirchhoff's rule \eqref{eq:Kirchhoff} for $k_{ij}$, and 
the definition \eqref{eq:def-nabla-v} of $\nabla v$, we deduce 
\begin{align}
& \sqrt{l_{i_0}}\sum_{(i,j)\in\vec{E}}\kappa_{ij}(v_i-v_j)
=\sqrt{l_{i_0}}\sum_{(i,j)\in\vec{E}}(\kappa_{ij}-\kappa_{ji})v_i 
=\sum_{(i,j)\in\vec{E}}(k_{ij}-k_{ji})v_i \cr
=& \sum_{i\in V}v_i\sum_{\substack{j\in V:\\ \{i,j\}\in E}}(k_{ij}-k_{ji})
= \sum_{i\in V}v_i (\delta_{i_0}(i)-\delta_{i_1}(i)) =\nabla v.
\label{eq:with-kirchhoff-simplified}
\end{align}
Combining \eqref{eq:sum-omega-v-zero} and \eqref{eq:with-kirchhoff-simplified}
with \eqref{eq:klogW} yields  
\begin{align}
\label{eq:prod-powers-k}
\prod_{(i,j)\in\vec E} \left(\frac{W_{ij}l_ie}{2k_{ij}}\right)^{k_{ij}}
=\exp\Big(l_{i_0}\sum_{(i,j)\in\vec E}\omega_{ij} +\nabla v
-\sum_{(i,j)\in\vec E}\frac{\kappa_{ij}^2}{2\omega_{ij}}
+O_{M,G}\left(\sigma^{-1/2}\right)\Big). 
\end{align} 
Inserting \eqref{eq:prod-sqrt-k} and \eqref{eq:prod-powers-k} into 
\eqref{eq:part-of-density-for-one-kij} and using the definition 
\eqref{eq:def-pp} of $\pp$ yields 
\begin{align}
\label{eq:intermediate-summary}
\pp(k,l,\vec T) 
= & \exp\Big(l_{i_0}\sum_{(i,j)\in\vec E}\omega_{ij} + \nabla v
-\sum_{(i,j)\in\vec E}\frac{\kappa_{ij}^2}{2\omega_{ij}}\Big)\cr
& \cdot\left(\frac{l_{i_0}}{2\pi}\right)^{|E|}\prod_{\{i,j\}\in E}\omega_{ij}
\prod_{(i,j)\in\vec T}\frac{1}{l_i}\prod_{(i,j)\in\vec{E}\setminus\vec T}\frac{1}{k_{ij}} 
\cdot(1+O_{M,W,G}(\sigma^{-1/2})).
\end{align}
Using \eqref{eq:kij-asymptotics} and the fact that $\vec T$ is a spanning tree 
directed towards $i_1$, we obtain 
\begin{align}
\label{eq:tree}
\prod_{(i,j)\in\vec T}\frac{1}{l_i}\prod_{(i,j)\in\vec{E}\setminus\vec T}\frac{1}{k_{ij}}
=l_{i_0}^{-2|E|}\prod_{i\in V\setminus\{i_1\}}e^{-2v_i}
\prod_{(i,j)\in\vec{E}\setminus\vec T}\frac{1}{\omega_{ij}}\cdot(1+O_{M,W,G}(\sigma^{-1/2})). 
\end{align}
Combining this with \eqref{eq:intermediate-summary} yields
\begin{align}
\pp(k,l,\vec T) 
= & \frac{1}{(2\pi l_{i_0})^{|E|}} 
\exp\Big(l_{i_0}\sum_{(i,j)\in\vec E}\omega_{ij} + \nabla v
-\sum_{(i,j)\in\vec E}\frac{\kappa_{ij}^2}{2\omega_{ij}}\Big) \nonumber\\
& \cdot\prod_{i\in V\setminus\{i_1\}}e^{-2v_i}
\prod_{(i,j)\in\vec{E}\setminus\vec T}\frac{1}{\omega_{ij}}
\prod_{\{i,j\}\in E}\omega_{ij}\cdot(1+O_{M,W,G}(\sigma^{-1/2})). 
\label{eq:first-summary}
\end{align}
Using $\nabla v=-v_{i_1}$, the claim follows.
\end{proof}

\section{Proof of Lemma \ref{le:integral-kappa-single}: Gaussian
integral over currents}
\label{app:B}

We endow every undirected edge in $E$ with a counting direction. 
For $\kappa\in\cH$ and $i,j\in V$ such that $\{i,j\}\in E$, we introduce the 
following variables
\begin{align}
\label{eq:def-I-J}
I_{ij}=\frac{1}{\sqrt 2}(\kappa_{ij}-\kappa_{ji}), \quad
J_{ij}=\frac{1}{\sqrt 2}(\kappa_{ij}+\kappa_{ji}).
\end{align}
Note that $I$ is antisymmetric ($I_{ij}=-I_{ji}$) and $J$ is symmetric 
($J_{ij}=J_{ji}$). 
Recall that $\vec T_0$ denotes a directed reference spanning tree and $T_0$ 
its undirected version. 
Recall that the restriction map $\iota:\cH\to\R^{\vec E\setminus\vec T_0}$ is an 
isomorphism. In other words, 
the components $\kappa_{\alpha\beta}$, $(\alpha,\beta)\in\vec E\setminus\vec T_0$, of 
$\kappa\in\cH$ can be chosen arbitrarily while all other $\kappa_{ij}$, 
$(i,j)\in\vec T_0$, are determined by the first. We define now a linear map
$L:\R^{\vec E\setminus\vec T_0}\to\R^{E\setminus T_0}\times\R^E$. Given 
$\tilde\kappa\in\R^{\vec E\setminus\vec T_0}$, we set $\kappa=\iota^{-1}(\tilde\kappa)$
and 
\begin{align}
L(\tilde\kappa)
=\left((I_{ij}(\kappa))_{\{i,j\}\in E\setminus T_0},(J_{ij}(\kappa))_{\{i,j\}\in E}
\right), 
\end{align}
where the vertices $i$ and $j$ in $I_{ij}$ are ordered with respect to the 
counting direction of the edge $\{i,j\}$, in order to have no ambiguity with 
the sign of $I_{ij}$.  
We claim that the determinant of $L$ equals $\pm 2^{\frac{|V|-1}{2}}$. In other
words, this yields the change of measure
\begin{align}
\label{cm}
\iota[d\kappa_\cH]=
\prod_{(i,j)\in\vec{E}\setminus\vec{T}_0}d\kappa_{ij}
=2^{-\frac{|V|-1}{2}}\prod_{\{i,j\}\in E\setminus T_0}dI_{ij}\prod_{\{i,j\}\in E}dJ_{ij}. 
\end{align}
Indeed, with an appropriate order of indices, the matrix associated to $L$ is 
given by 
\begin{align}
\left(\begin{array}{c}
\left(\dfrac{\partial J_{ij}}{\partial\ka_{\al\be}}\right)_{\substack{\{i,j\}\in T_0,\\ (\al,\be)\in\vec{E}\setminus\vec{T}_0}}
\vspace{3mm}\\ 
\left(\dfrac{\partial}{\partial\ka_{\al\be}}\left(\begin{array}{c}I_{ij}\\J_{ij}\end{array}\right)\right)_{\substack{\{i,j\}\in E\setminus T_0,\\ (\al,\be)\in\vec{E}\setminus\vec{T}_0}}
\end{array}
\right),
\end{align}
which can be written as follows, by an appropriate choice of order on the second index:
\begin{align}
\left(\begin{array}{cc}
\left(\dfrac{\partial J_{ij}}{\partial\ka_{\al\be}}\right)_{\substack{\{i,j\}\in T_0,\\ (\be,\al)\in\vec{T}_0}}&\left(\dfrac{\partial J_{ij}}{\partial\ka_{\al\be}}\right)_{\substack{\{i,j\}\in T_0,\\ (\al,\be):\,\{\al,\be\}\in E\setminus T_0}}
\vspace{3mm}
\\ 
\left(\dfrac{\partial}{\partial\ka_{\al\be}}\left(\begin{array}{c}I_{ij}\\J_{ij}\end{array}\right)\right)_{\substack{\{i,j\}\in E\setminus T_0,\\ (\be,\al)\in\vec{T}_0}}
&\left(\dfrac{\partial}{\partial\ka_{\al\be}}\left(\begin{array}{c}I_{ij}\\J_{ij}\end{array}\right)\right)_{\substack{\{i,j\}\in E\setminus T_0,\\ (\al,\be):\,\{\al,\be\}\in E\setminus T_0}}
\end{array}
\right).
\end{align}
We order the indices $(\al,\be)$ with $\{\al,\be\}\in E\setminus T_0$ in the second 
block column successively by groups of two, associated to each nonoriented edge 
$\{\al,\be\}\in E\setminus T_0$, taking first the oriented edge corresponding to 
the arbitrary counting direction. We claim that 
the Jacobian matrix above takes the following block triangular form:
\begin{align}
\label{eq:jacobian-L} 
\left(\begin{array}{cc}
\sqrt{2}\, \id_{|T_0|\times|T_0|}& \left(*\right)_{|T_0|\times 2|E\setminus T_0|}\vspace{3mm}
\\ 
(0)_{2|E\setminus T_0|\times|T_0|}
&\dfrac{1}{\sqrt{2}}
\left(
\begin{array}{cc}
1 &-1\\
1 & 1
\end{array}
\right) \otimes \id_{|E\setminus T_0|\times|E\setminus T_0|}
\end{array}
\right),
\end{align}
In order to see why the first block column takes the claimed form, 
let $(\beta,\alpha)\in\vec T_0$ and take 
$\tilde\kappa=(\delta_\beta(i)\delta_\alpha(j))_{(i,j)\in\vec E\setminus\vec T_0}
\in\R^{\vec E\setminus\vec T_0}$. Then, $\kappa=\iota^{-1}(\tilde\kappa)$ is 
given by $\kappa_{\alpha\beta}=\kappa_{\beta\alpha}=1$ and $\kappa_{ij}=0$ 
otherwise. This implies 
$I_{ij}(\kappa)=0$ for all $\{i,j\}\in E$, 
$J_{\alpha\beta}(\kappa)=\sqrt 2$, and $J_{ij}(\kappa)=0$ otherwise. 
This explains the blocks $\sqrt 2\id$ and $0$.
The expression for the lower right block in the matrix \eqref{eq:jacobian-L} 
follows from the definition \eqref{eq:def-I-J} using that 
$\kappa_{\alpha\beta}$ with $(\alpha,\beta)\in\vec E\setminus\vec T_0$ 
are linearly independent variables.
We conclude
\begin{align}
|\det L|=\sqrt{2}^{|T_0|}\left[\det\dfrac{1}{\sqrt{2}}
\left(
\begin{array}{cc}
1 &-1\\
1 & 1
\end{array}
\right)\right]^{|E\setminus T_0|}
=2^{\frac{|T_0|}{2}}\cdot 1^{|E\setminus T_0|} 
=2^{\frac{|V|-1}{2}} 
,
\end{align} 
in other words \eqref{cm} holds.

Note that for all $\{i,j\}\in E$, 
\begin{align}
I_{ij}^2+J_{ij}^2
=\frac12\left((\kappa_{ij}-\kappa_{ji})^2+(\kappa_{ij}+\kappa_{ji})^2\right)
=\kappa_{ij}^2+\kappa_{ji}^2.
\end{align}
Consequently, we obtain
\begin{align}
& \int_\cH
\exp\Big(-\sum_{(i,j)\in\vec E} \frac{\kappa_{ij}^2}{2\omega_{ij}}
\Big) \, d\kappa_{\cH}
\cr
= & 2^{-\frac{|V|-1}{2}}\int_{\R^{E\setminus T_0}}\prod_{\{i,j\}\in E}
\exp\left(-\frac{I_{ij}^2}{2\omega_{ij}}\right)
\prod_{\{i,j\}\in E\setminus T_0}dI_{ij}\prod_{\{i,j\}\in E}
\int_\R\exp\left(-\frac{J_{ij}^2}{2\omega_{ij}}\right)dJ_{ij}
\cr
= & 2^{\frac{|E|-|V|+1}{2}}\pi^{\frac{|E|}{2}}\int_{\R^{E\setminus T_0}}\prod_{\{i,j\}\in E}
\exp\left(-\frac{I_{ij}^2}{2\omega_{ij}}\right)
\prod_{\{i,j\}\in E\setminus T_0}dI_{ij}\,
\prod_{\{i,j\}\in E}\sqrt{\omega_{ij}}.
\end{align}
For $e\in E\setminus T_0$, let $c_e$ be the unique oriented cycle in 
$T_0\cup\{e\}$ containing the edge $e$ in its counting direction. 
For another edge $g\in E$, let $\sigma_{eg}$ be $+1$ if $e$ and $g$ 
appear in the same counting direction in $c_e$, $-1$ if $e$ and $g$ 
appear in opposite counting directions in $c_e$, and $0$ if $g\notin c_e$.  
We define the matrix $B=(B_{ef})_{e,f\in E\setminus T_0}$ by 
\begin{align}
B_{ee}=\sum_{g\in c_e}\frac{1}{\omega_g}, \quad
B_{ef}=\sum_{g\in c_e\cap c_f}
\frac{\sigma_{eg}\sigma_{fg}}{\omega_g}\quad\text{for }e\neq f. 
\end{align}
Let $I=(I_{ij})_{\{i,j\}\in E\setminus T_0}$ denote the restriction to $E\setminus T_0$.
Note that the full vector $(I_{ij})_{\{i,j\}\in E}$ can be retrieved from its
restriction $I$ using the formula
$I_f=\sum_{e\in E\setminus T_0}\sigma_{ef}I_e$. This formula implies 
\begin{align}
\sum_{\{i,j\}\in E}\frac{I_{ij}^2}{2\omega_{ij}}
=\sum_{e\in E\setminus T_0}\sum_{f\in E\setminus T_0}I_eI_f
\sum_{g\in E}\frac{\sigma_{eg}\sigma_{fg}}{2\omega_g}
= \frac{1}{2} I^tBI.
\end{align}
We abbreviate $dI=\prod_{\{i,j\}\in E\setminus T_0}dI_{ij}$. Using the formula
for the determinant of $B$ given as last displayed formula on page 20 of
\cite{Keane-Rolles2000}, which is based on the matrix-tree theorem, we obtain 
\begin{align}
& \int_{\R^{E\setminus T_0}}\prod_{\{i,j\}\in E}
\exp\left(-\frac{I_{ij}^2}{2\omega_{ij}}\right)
\prod_{\{i,j\}\in E\setminus T_0}dI_{ij}
= \int_{\R^{E\setminus T_0}}e^{-\frac{1}{2}I^tBI}dI \nonumber\\ 
= &\frac{(2\pi)^{\frac{|E|-|V|+1}{2}}}{\sqrt{\det B}}
=  (2\pi)^{\frac{|E|-|V|+1}{2}}\frac{\prod_{\{i,j\}\in E}\sqrt{\omega_{ij}}}{
\sqrt{\sum_{S\in\T}\prod_{\{i,j\}\in S}\omega_{ij}}}.
\end{align}
We conclude that
\begin{align}
\int_\cH
\exp\Big(-\sum_{(i,j)\in\vec E} \frac{\kappa_{ij}^2}{2\omega_{ij}}
\Big) \, d\kappa_{\cH}
= & 2^{\frac{|E|-|V|+1}{2}}\pi^{\frac{|E|}{2}}
(2\pi)^{\frac{|E|-|V|+1}{2}}\frac{\prod_{\{i,j\}\in E}\sqrt{\omega_{ij}}}{
\sqrt{\sum_{S\in\T}\prod_{\{i,j\}\in S}\omega_{ij}}}
\prod_{\{i,j\}\in E}\sqrt{\omega_{ij}} \nonumber\\
= & 2^{|E|-|V|+1}\pi^{|E|-\frac{|V|-1}{2}}\frac{\prod_{\{i,j\}\in E}\omega_{ij}}{
\sqrt{\sum_{S\in\T}\prod_{\{i,j\}\in S}\omega_{ij}}}.
\end{align}

\section{Review of used results}
\label{sec:app-review-results}

\paragraph{Normalization of the supersymmetric hyperbolic sigma model.}
\begin{fact}
\label{fact:normalization-susy} 
\begin{align}
\int_{\Omega_{i_0}}
\exp\Big(\sum_{\{i,j\}\in E} W_{ij}\left(
1-\cosh(v_i-v_j)\right)\Big) 
\sqrt{\sum_{S\in\T}\prod_{\{i,j\}\in S}W_{ij}e^{v_i+v_j}}\prod_{i\in V\setminus\{i_0\}}
\frac{e^{-v_i}}{\sqrt{2\pi}}\,dv_i =1. 
\nonumber
\end{align}
\end{fact}
We refer to formula (3) in Theorem 2 in \cite{sabot-tarres-zeng2015} with 
$\phi_i=1$ for all $i$, which gives
a new proof of this normalization using an interpretation of $u$ through the Green's 
function of random Schr\"odinger operators. 

The first proof of Fact \ref{fact:normalization-susy} was given in 
\cite{disertori-spencer-zirnbauer2010}. It heavily uses the supersymmetry of the model,
which is only seen in the version of the $\htwo$ model with Grassmann variables, cf.\ 
formula (5.1) and Proposition 2 in appendix C of 
\cite{disertori-spencer-zirnbauer2010}. Note that the reference point 
$i_0$ is not mentioned explicitly in \cite{disertori-spencer-zirnbauer2010}, but the
pinning strengths $\varepsilon_i$ in that paper play the role of the weights $W_{ii_0}$ 
connecting any vertex $i\in V\setminus\{i_0\}$ to the reference vertex $i_0$, 
while the coupling constants $\beta J_{ij}$ play the role of all other weights $W_{ij}$, 
$\{i,j\}\in E$ with $i,j\neq i_0$.

The link between the $\htwo$ model in horospherical coordinates, treated in 
\cite{disertori-spencer-zirnbauer2010}, and its tree version, 
used in the current paper, is given by the matrix tree theorem stated in formula 
(2.17) in \cite{disertori-spencer-zirnbauer2010}. More precisely, 
the sum $\sum_{T'\in\T}\prod_{\{i,j\}\in T'}W_{ij}e^{u_i+u_j}$ arises from 
the matrix tree theorem as the same determinant that occurs when integrating out 
the Grassmann variables 
$\overline\psi_i$ and $\psi_i$. A variant of this argument concerning only the 
marginal of $u$ with respect to $\mu_{i_0}^\susy$ is also described in 
\cite{disertori-spencer-zirnbauer2010}; see formula (1.4) in that paper for the 
statement. 

There is at least another proof of Fact \ref{fact:normalization-susy} that does not
use supersymmetry. In the paper \cite{sabot-tarres2012} of Sabot and Tarr\`es, the 
normalization comes from the fact
that the marginal in $u$ of the measure is interpreted as a probability distribution for 
a random variable associated to asymptotic behavior of the vertex-reinforced jump
process.

\paragraph{Density of paths.}
According to the first displayed formula on page 569 of \cite{sabot-tarres2016}, 
the probability that, at time $t$, the process $Z$ has followed a path
$Z_0=\pi_ 0$, $\pi_1$,
$\ldots$, $Z_t = \pi_n$ with jump times respectively in $[t_i,t_i+dt_i]$, 
$i=1,\ldots,n$, where $t_0=0<t_1 <\cdots < t_n < t =t_{n+1}$, is given by 
\begin{align}
\label{eq:density-sabot-tarres}
\exp\Big(\sum_{\{i,j\}\in E} W_{ij}\Big(1-\sqrt{1+2\ell_i}\sqrt{1+2\ell_j}\Big)\Big)
\prod_{i\in V\setminus\{i_1\}}\frac{1}{\sqrt{1+2\ell_i}} 
\prod_{i=1}^n W_{\pi_{i-1}\pi_j}dt_i.
\end{align}
This formula is precisely the cited formula in the special case of offset $\varphi=1$,
with the abbreviations from \cite{sabot-tarres2016} already inserted. 
Their time $t$ and local times $\ell_i$ differ from our time $\sigma=2t$ and local
times $l_i=2\ell_i$ by a factor of $1/2$ as can be seen by comparing formula 
\eqref{eq:def-D} with the expression for $D(s)$ on page 567 of 
\cite{sabot-tarres2016}. Rescaling all times by this factor $1/2$ yields
that the expression in \eqref{eq:density-sabot-tarres} equals
 \begin{align}
\exp\Big(\sum_{\{i,j\}\in E} W_{ij}\Big(1-\sqrt{1+l_i}\sqrt{1+l_j}\Big)\Big)
\prod_{i\in V\setminus\{i_1\}}\frac{1}{\sqrt{1+l_i}} 
\prod_{i=1}^n \frac{W_{\pi_{i-1}\pi_j}}{2}d\sigma_i.
\label{eq:density-s-t-2}
\end{align}
Formula \eqref{eq:prob-density-fixed-path-single} is just obtained from this 
density by integrating over an event. 
Formula \eqref{eq:density-s-t-2} in almost the same notation is also cited in 
formula (2) in Section~3.3 of \cite{zeng13}.

\paragraph{Counting paths.}
\begin{fact}
\label{eq:fact-counting-paths}
For all $i_0,i_1\in V$, $k\in\K^+_{i_0,i_1}$, and all $\vec T\in\vec\T_{i_1}$, one has
\begin{align}
\label{eq:number-paths-KR}
|\Pi_{i_0,i_1}(k,\vec T)|=
& \frac{\prod_{i\in V}k_i!}{\prod_{(i,j)\in\vec{E}}k_{ij}!}\cdot
\frac{\prod_{(i,j)\in \vec T}k_{ij}}{\prod_{i\in V\setminus\{i_1\}}k_i}
=k_{i_1}!\prod_{i\in V\setminus\{i_1\}}(k_i-1)!
\frac{\prod_{(i,j)\in \vec T}k_{ij}}{\prod_{(i,j)\in\vec{E}}k_{ij}!}
\end{align}
with $k_i$ defined in \eqref{eq:def-k-i}. 
\end{fact}

In Lemma 6 of \cite{Keane-Rolles2000}, the first equality in \eqref{eq:number-paths-KR}
is stated with an additional
summation over all $\vec T\in\vec\T$. However, the proof of this lemma is based on 
a combinatorial lemma (Lemma 5 of \cite{Keane-Rolles2000}) which is also applicable for 
any fixed $\vec T\in\vec\T_{i_1}$, without summation over $\vec T$. 
The idea behind this combinatorial lemma is the following: 
we attach to every vertex $i\in V$ a sequence of 
length $k_i$ consisting of directed edges $(i,j)$ such that vertex $j\in V$ occurs 
precisely $k_{ij}$ times and the union of the last edges from all $i\neq i_1$ 
yields $\vec T$. Starting at $i_0$ we construct a path by always traversing
the first edge which was not used earlier at the current location. 
This yields that $|\Pi_{i_0,i_1}(k,\vec T)|$ equals the following product 
over multinomial coefficients, counting the number of choices at every vertex $i\in V$:
\begin{align}
|\Pi_{i_0,i_1}(k,\vec T)|
=\prod_{i\in V} \frac{(k_i-1_{\{i\neq i_1\}})!}{\prod_{\substack{j\in V:\\ (i,j)\in\vec E}}
(k_{ij}-1_{\{(i,j)\in\vec T\}})!}
=\frac{k_{i_1}!\prod_{i\in V\setminus\{i_1\}}(k_i-1)!}{\prod_{(i,j)\in\vec T}(k_{ij}-1)!
\prod_{(i,j)\in\vec{E}\setminus\vec T}k_{ij}!}
\end{align}
\end{appendix}

\noindent{\bf Acknowledgements:} The authors would like to thank an 
anonymous referee and the associate editor for very constructive comments
helping us to improve the paper.

This work is supported by National Science Foundation of China (NSFC), 
grant No.\ 11771293, and by the Agence Nationale de la Recherche (ANR) 
in France, project MALIN, No.\ ANR-16-CE93-0003.

\end{document}